\documentclass[
paper=a4,
fontsize=10pt,
headings=small,
parskip=never,
toc=left
]{scrartcl}

\usepackage[T1]{fontenc}
\usepackage[utf8]{inputenc}

\usepackage[leqno]{amsmath}
\usepackage{amssymb}
\usepackage{amsthm}
\usepackage{mathtools}
\usepackage{latexsym}
\usepackage{bussproofs}
\usepackage{accents}
\usepackage{graphicx}
\usepackage{textcomp}

\usepackage{url}
\DeclareUrlCommand\email{\urlstyle{rm}}

\usepackage[rightmargin=0pt,vskip=\baselineskip]{quoting}%

\renewenvironment{quote}[1]%
{\begin{quoting}#1}
{\end{quoting}}

\renewenvironment{quotation}[1]%
{\begin{quoting}#1}
{\end{quoting}}

\usepackage{leading}
\leading{13pt}

\setkomafont{disposition}{\normalfont}

\usepackage{ragged2e}


\wlog{\the\footnotesep}

\makeatletter
\renewcommand\deffootnote[4][]{%
  \long\def\@makefntext##1{%
    \setlength{\@tempdimc}{#3}%
    \def\@tempa{#1}\ifx\@tempa\@empty
      \@setpar{\@@par
        \@tempdima = \hsize
        \addtolength{\@tempdima}{-#2}%
        \parshape \@ne #2 \@tempdima}%
    \else
      \addtolength{\@tempdimc}{#2}%
      \addtolength{\@tempdimc}{-#1}%
      \@setpar{\@@par
        \@tempdima = \hsize
        \addtolength{\@tempdima}{-#1}%
        \@tempdimb = \hsize
        \addtolength{\@tempdimb}{-#2}%
        \parshape \tw@ #1 \@tempdima #2 \@tempdimb
      }%
    \fi
    \par
    \parindent\@tempdimc\noindent
    \ftn@font\hbox to \z@{\hss\@@makefnmark}\usebox\strutbox##1
  }%
  \def\@@makefnmark{\hbox{\ftnm@font{#4}}}%
}%
\makeatother

\deffootnote{\footnotemarkspace}
  {\parindent}%
  {\makebox[\footnotemarkspace][r]{\thefootnotemark\ }} %

\usepackage{geometry}%
\newcommand*\GetTextWidth[3][\normalfont\normalsize]{{#1%
    \settowidth{#2}{abcdefghijklmnopqrstuvwxyz}%
    \setlength{#2}{0.03193#2}%
    \addtolength{#2}{0.44961pt}%
    \setlength{#2}{#3#2}%
    \global#2=#2}}

\newlength\bringhurstwdt
\GetTextWidth{\bringhurstwdt}{72}
\usepackage[german,dutch,french,italian,english]{babel}
\useshorthands{"}
\addto\extrasenglish{\languageshorthands{ngerman}}

\newcommand{\dutch}[1]{\foreignlanguage{dutch}{#1}}
\newcommand{\english}[1]{\foreignlanguage{english}{#1}}
\newcommand{\french}[1]{\foreignlanguage{french}{#1}}
\newcommand{\german}[1]{\foreignlanguage{german}{#1}}
\newcommand{\italian}[1]{\foreignlanguage{italian}{#1}}
\newcommand{\weg}[1]{}

\usepackage{graphics}
\newcommand{\klo}{\mathrel{<\mkern-9mu\raisebox{0.6pt}{\scalebox{0.7}{$\circ$}}}}
\newcommand{\gro}{\mathrel{\raisebox{0.6pt}{\scalebox{0.7}{$\circ$}}\mkern-9mu>}}

\newcommand{\apart}{\ensuremath\mathrel{\#}}

\DeclareUnicodeCharacter{03B1}{\ensuremath{\alpha}}
\DeclareUnicodeCharacter{03B2}{\ensuremath{\beta}}
\DeclareUnicodeCharacter{03B3}{\ensuremath{\gamma}}
\DeclareUnicodeCharacter{03B5}{\ensuremath{\epsilon}}
\DeclareUnicodeCharacter{03B4}{\ensuremath{\delta}}
\DeclareUnicodeCharacter{03B7}{\ensuremath{\eta}}
\DeclareUnicodeCharacter{03C0}{\ensuremath{\pi}}
\DeclareUnicodeCharacter{03C4}{\ensuremath{\tau}}
\DeclareUnicodeCharacter{03C9}{\ensuremath{\omega}}

\newcommand{\numberset}[1]{\mathbb{#1}} 

\newcommand{\numnat}{\numberset{N}}
\newcommand{\numrat}{\numberset{Q}}
\newcommand{\numreal}{\numberset{R}}

\newcommand{\st}{\ensuremath{\text{standard}}}
\newcommand{\forallst}{\forall^{\,\text{st}}}
\newcommand{\forallstfin}{\forall^{\,\text{st\,fin}}}
\newcommand{\forallstinf}{\forall^{\,\text{st\,inf}}}
\newcommand{\existsst}{\exists^{\,\text{st}}}
\newcommand{\existsstfin}{\exists^{\,\text{st\,fin}}}

\newcommand{\breakingcomma}{%
  \begingroup\lccode `~=`,
 \lowercase{\endgroup\expandafter\def\expandafter~\expandafter{~\penalty0 }}}

\newcommand{\formatabr}[1]{{\MakeUppercase{#1}}}

\newcommand{\cnrs}{\formatabr{cnrs}}

\newcommand{\iv}{\formatabr{iv}}
\newcommand{\snd}{\formatabr{snd}}

\newcommand*\footnotemarkspace{1.5em} 

\makeatletter
\g@addto@macro \normalsize {%
\setlength{\abovedisplayskip}{\baselineskip}%
\setlength{\belowdisplayskip}{\baselineskip}%
\setlength{\abovedisplayshortskip}{\baselineskip}%
\setlength{\belowdisplayshortskip}{\baselineskip}%
\setlength{\jot}{0pt}
}
\makeatother

\renewcommand{\newblock}{ }

\usepackage{natbib}

\DeclareOldFontCommand{\bf}{\normalfont\bfseries}{\mathbf}

\usepackage{hyphenat}

\newtheorem{theorem}{Theorem}
\newtheorem{lemma}[theorem]{Lemma}
\newtheorem{corollary}[theorem]{Corollary}
\newtheorem{pargmt}[theorem]{Plausibility argument}
\newtheorem{proof}[theorem]{Proof}
\newtheorem{wce}[theorem]{Weak counterexample}

\begin{document}

\setlength{\RaggedRightParindent}{1.5em}
\RaggedRight

{\Large
\begin{center}
Weyl and intuitionistic infinitesimals\end{center}
\par}

{\renewcommand*{\thefootnote}{\fnsymbol{footnote}}

\begin{center}
Mark van Atten$^*$\footnotetext[1]{\snd{} (\cnrs{}\,/\,Paris \iv{}),
1 rue Victor Cousin,
75005 Paris,
France. 
\mbox{vanattenmark@gmail.com}.}\\
\today
\end{center}
}
\setcounter{footnote}{0}

\begin{center}
dedicated to the memory of Richard Tieszen, 1951–2017
\end{center}

\deffootnote{\footnotemarkspace}
  {\parindent}%
  {\makebox[\footnotemarkspace][r]{{\thefootnotemark.\ }}} 

\begin{abstract}
\noindent As Weyl was interested in infinitesimal analysis and
for some years embraced Brouwer’s intuitionism,
which he continued to see as an ideal
even after he had 
convinced himself that
it is a practical necessity for science
to go beyond intuitionistic mathematics,
this note presents some remarks on
infinitesimals
from a Brouwerian perspective.
After an introduction and a look at Robinson’s and
Nelson’s approaches to classical nonstandard analysis,
three desiderata 
for an intuitionistic construction of infinitesimals 
are extracted from Brouwer’s writings.
These cannot be met,
but in explicitly Brouwerian settings
what might in different ways be called approximations to infinitesimals
have been developed by
early Brouwer,
Vesley,
and Reeb.
I conclude that perhaps Reeb’s approach,
with its Brouwerian motivation for accepting
Nelson’s classical formalism, 
would
have suited Weyl best.
\end{abstract}

\section{Introduction}
In the 1920s,
Weyl
doubted
that non-Archimedean number systems
could serve to develop 
real analysis with infinitesimals.%
\footnote{There are similar contemporary statements 
in e.g.~\citealt[116–117]{Fraenkel1929}
and 
\citealt[208]{Skolem1929}.
There are many studies of the history of
infinitesimals,
e.g.,
for the early history,
\citealt{Baron1969,Mancosu1999,Schubring2005}.
See
\citealt{FletcherHrbacekKanoveiKatzLobrySanders2017}
for an overview of later approaches.}
In his
‘\textit{Philosophy of Mathematics and Natural Science}’ of 1926,
Weyl advanced the following argument:
\begin{quotation}
But once the limit concept has been grasped,
it is seen to render the infinitely small superfluous.
Infinitesimal analysis proposes to draw conclusions
by integration from the behavior in the infinitely small,
which is governed by elementary laws,
to the behaviour in the large …
If the infinitely small is not interpreted ‘potentially’
here,
in the sense of the limiting process,
then the one has nothing to do with the other,
the processes in infinitesimal and in finite dimensions
become independent of each other,
the tie which binds them together is cut.
Here Eudoxus undoubtedly saw right
[in stating what became known as the Archimedean axiom].
(\citealt[35–36]{Weyl1926}/\citealt[44–45]{Weyl1949}.)
\end{quotation}

In the 1960s Abraham Robinson showed how to handle this fully
in ZFC and classical model theory,
introducing non-standard analysis
with its transfer principle.
This encouraged Edward Nelson
to develop
an 
approach likewise based on ZFC
but
axiomatic.
Constructive versions
of their theories have also been developed
(yielding,
as expected,
in general somewhat weaker principles).
But
Weyl for some years embraced Brouwer’s intuitionism,
and continued to see it as a philosophical ideal
even after he had 
convinced himself that
it is a practical necessity for science
to go beyond intuitionistic mathematics
and adopt a formalist attitude.
As Hans Freudenthal wrote in his obituary of his friend Weyl,  
\begin{quote}
He decides in favour of Brouwer’s intuitionistic explanation of mathematics,
but,
averse to system-building,
he spurns Brouwer’s method that aims for generality.
That is no treason,
then the essence of his mathematics was and remained intuitionistic.
(\citealt{Freudenthal1955},
as quoted in
\citealt[p.880n44, trl.~mine]{Dalen2005})
\end{quote}

After a brief look at
classical approaches
and
constructivisations of them,
I will
discuss
three attempts at introducing infinitesimals that were proposed as
not only constructive but,
in different ways,  
explicitly Brouwerian:
one by early Brouwer,
one by Vesley based on later Brouwer,
and one by Reeb.

These will not be compared with more recent
constructive versions of formalised nonstandard analysis as to their mathematical merits.%
\footnote{%
\label{constructivepapers}
See, e.g.,
\citealt{Palmgren1995},
\citealt{Berg.Briseid.Safarik2012},
\citealt{Ferreira.Gaspar2015},
\citealt{Sanders2018},
\citealt{Berg.Sanders2017},
and
\citealt{Dinis.Gaspar2018}.
Such approaches are of particular interest for ‘proof mining’,
which is the extraction of explicit information from proofs
(e.g., explicit bounds for existential quantifiers,
or rates of convergence).}
They are discussed
primarily because  they stay close to Weyl’s own interest
in both Brouwerian intuitionism and infinitesimals,
but to some extent also for their own sake,
as they remain among the lesser known.
It would be very interesting to see a further mathematical development
of Vesley’s approach,
as it is,
by design,
fully integrated in Brouwerian analysis with choice sequences;
and Reeb’s
Brouwerian look at Nelson’s classical nonstandard analysis
is of considerable philosophical interest,
and would have been so,
I suspect,
to Weyl.

\section{Classical and constructive approaches}%
\label{SectionClassicalConstructive}Robinson’s approach is
model-theoretical.
It exploits the fact that
a theory may have non-isomorphic models
to distinguish which would require
a stronger logic.
That Peano Arithmetic with first-order logic has a model that is not
isomorphic to the intended one was shown,
by classical means, 
in \citealt{Skolem1934}.
Robinson established the same 
for the theory of analysis with first-order logic.
The field of real numbers \(\numreal\) has an Archimedean order:
\begin{equation*}
\label{ArchOrder}
\forall x,y\in\numreal(x < y \rightarrow 
\exists n\in\numnat(n \cdot x > y))
\end{equation*}
The non-standard infinite elements
in Robinson’s non-standard model
give rise to a
non-Archimedean order,
and he proceeded in such a way
that
the larger structure
is again a field,
\(\numreal^*\).
By definition,
a field
\(F\)
contains,
for each of its elements except 0,
also its 
multiplicative inverse:
\begin{equation*}
\label{MultInv}
\forall x\in F
(x \neq 0 \rightarrow 
\exists y\in F
(x \cdot y = 1))
\end{equation*}
The multiplicative inverses
of the infinitely large elements in
$\numreal^*$ 
are infinitesimals.

Robinson’s method enabled him to show
\begin{equation*}
\numreal^* \vDash \phi \Leftrightarrow \numreal \vDash \phi
\label{TP}
\end{equation*}
for 
\(\phi\)
limited to first-order formulas.
First-order logic is not strong enough to distinguish between
non-isomorphic models of the theory of real numbers.
The equivalence~\eqref{TP}
is called the Transfer Principle
because it states
that,
for first-order
$\phi$,
truth in $\numreal^*$
transfers to
truth in $\numreal$,
and vice versa.

Robinson observed 
that
the philosophical importance of the Transfer Principle
is that,
within the limitation to stay with first-order logic,
it is a mathematical rendering of Leibniz’s Continuity Principle: 
\begin{quotation}
G.W.~Leibniz argued that the theory of infinitesimals
implies the introduction of ideal numbers which might be
infinitely small or infinitely large
compared with the real numbers
but which were
\emph{to possess the same properties as the latter}.%
\footnote{[Note MvA: An example of Leibniz’ saying this is found in his well-known letter to Varignon of February 2, 1702
\citep[pp.93-94]{Leibniz1859}.]}

However,
neither he nor his disciples and successors were able
to give a rational development leading up to a system of this sort.
As a result,
the theory of infinitesimals gradually fell into disrepute
and was replaced eventually by the classical theory of limits.

It is shown in this book that Leibniz’ ideas can be fully vindicated
and that they lead to a novel and fruitful approach to classical
Analysis and to many other branches of mathematics.
\citep[p.2, original emphasis]{Robinson1966}
\end{quotation}

A 
largely 
constructive,
version of non-standard analysis had been
developed by Schmieden and Laugwitz 
\citeyearpar{Schmieden.Laugwitz1958}
even before Robinson’s classical work.
Their non-standard objects can be considered to be constructive
–~they are infinite sequences of arbitrary rational numbers~–
but classical reasoning is used to reason about
their properties.
Also  their transfer principle was  limited;
their 
\(\numreal^*\) 
was not a field but only
a partially ordered ring.
Inspired by Schmieden and Laugwitz,
a fully constructive counterpart to Robinson’s work
has been
developed by Erik Palmgren.
It turns out that
full Transfer would require a change in the logic:
the following is the argument presented in
\citealt[p.234]{Palmgren1998}.

\begin{theorem}[Moerdijk and Palmgren]
A full constructivisation of Robinson’s Transfer Principle
demands
a nonstandard interpretation of the
logical symbols.
\end{theorem}
\begin{proof}
Assume that
$\forall x P(x)$
is an as yet undecided formula in arithmetic,
where 
$P$
contains no unbounded quantifiers.
(Goldbach’s conjecture is of this form.)
We can decide all instances
up to any given bound,
so
\begin{equation*}
\numnat \vDash 
\forall m(\forall n<m P(n) \vee \exists n<m \neg P(n))
\end{equation*}
Suppose now we have a Transfer Principle:
\begin{equation*}
\label{ArithmeticTransfer}
\numnat^* \vDash \phi \Leftrightarrow \mathbb{N} \vDash \phi
\end{equation*}
Let 
$m \in \numnat^*$
be infinite.
By transfer from
$\numnat$
to
$\numnat^*$,
either for all 
$n$ we have
$\numnat^* \vDash P(n*)$,
where
$n*$
is the image of 
$n\in \numnat$ 
in
$\numnat^*$,
or
$\numnat^* \vDash \exists n<m \neg P(n)$.
Transfer back from 
$\numnat^*$
to
$\numnat$
yields
$\numnat \vDash 
\forall x P(x) \vee \neg\forall x P(x)$.
This contradicts our assumption.
\end{proof}

Inspired by Robinson’s theory,
but wishing to reconstruct it in syntactic
rather than model-theoretical terms,
Edward Nelson proposed
Internal Set Theory (IST)
\citep{Nelson1977}.%
\footnote{For an introduction to Nelson’s approach, 
see his more explanatory first chapter of a projected
book \citealt{Nelson2002} and \citealt{Robert1988}.}
Another syntactic approach had
been invented just before him
by Hrbaček,
and closely related
work 
(but incompatible with ZFC)
by Vopenka had begun even earlier;
for the historical details,
I refer to the rich footnote 7
in 
\citealt[p.vii]{Kanovei.Reeken2004}.
I will here look at IST,
and in some detail, 
not so much because
it is the best known and most used syntactical nonstandard analysis,
but 
because it was
part of Reeb’s Brouwerian approach
that we will see in section~\ref{SectionReeb}.
It should be mentioned,
however,
that further development of
both Nelson’s and Hrbaček’s
work
have led to theories with
more attractive metamathematical properties.%
\footnote{The culmination of this is
HST
in 
\citealt{Kanovei.Reeken2004};
but it does not contain full ZFC as a proper part.
See also footnote~\ref{istnotunique},
below.}

Instead of enriching the ontology by adding nonstandard objects to the set of classical real numbers,
Nelson  enriches the language we have to talk about the latter.%
\footnote{It is, 
in fact, 
possible to look at Robinson’s nonstandard analysis in an entirely
formalistic way, 
and take it not to introduce new objects,
but new ways of deducing theorems.
Robinson points this out at the very end of his book:
\begin{quote}
Returning now to the theory of this book,
we observe that it is presented,
naturally,
within the framework of contemporary Mathematics,
and thus appears to affirm the existence of all sorts of
infinitary entities.
However,
from a formalist point of view we may look at our theory
syntactically and may consider that what we have done is to
introduce new deductive procedures rather than new
mathematical entities.
\citep[p.282]{Robinson1966}
\end{quote}
I have not highlighted this in the main text,
so as to be able to show the contrast between the model-theoretical and the syntactical approaches.}
The idea is that,
from a formal point of view,
a distinction between standard and nonstandard numbers
useful for the development of analysis
can already be made within the set of classical real numbers;
what matters is that this is done in such a way that the
right formulas become provable.

Nelson adds an undefined predicate
‘\st’ 
to the language of ZFC
and adds three axioms to the theory
that regulate its use;
just as in ZFC,
the relation
\(\in\)
is undefined.
As the phrase goes, 
‘its meaning is implicitly defined by the axioms’;
but
course that is not a specification of a meaning in the sense
of presenting a construction method or a meaning explanation in the
sense of, for example, Dummett and Martin-Löf.

In the axioms for
‘\st’,
the following shorthand is used:%
\footnote{The predicates ‘finite’ and ‘infinite’ are defined as usual,
in terms of the presence or absence of a bijection between
\(x\)
and 
the set
\(\{m \mid m < n\}\)
for some
\(n \in \numnat\).}
\begin{align*}
\existsst x \phi(x) & \quad \text{ for } \quad \exists x(x \ \st \wedge \phi(x))\\
\forallst x \phi(x) & \quad \text{ for } \quad \forall x(x \ \st \rightarrow  \phi(x))\\
\forallstfin x \phi(x) & \quad \text{ for } \quad \forall x((x \ \st  \wedge  x \text{ finite})\rightarrow \phi(x))\\
\forallstinf x \phi(x) & \quad \text{ for } \quad \forall x((x \ \st  \wedge  x \text{ infinite}) \rightarrow \phi(x))
\end{align*}
Then the new axioms
are introduced,
formally,
and without pausing to motivate them.
Formulas not containing the predicate ‘\st’ are 
said to be internal
(namely, to ZFC),
those containing it external:
\begin{quotation}
The axioms of IST are the axioms of ZFC together with three additional axiom schemes which we call the transfer principle (T), 
the principle of idealization (I), 
and the principle of standardization (S). 
They are as follows.

Let 
\(A(x, t_1,\dots, t_k)\)
be an internal formula with free variables 
\(x, t_1,\dots, t_k\)
and no other free variables. Then
\[ 
\forallst t_1\dots \forallst t_k(\forallst x A(x, t_1,\dots, t_k)
\rightarrow
\forall x A(x, t_1,\dots, t_k)) \tag{T}
\]
Let 
\(B(x,y)\) 
be an internal formula with free variables 
\(x,y\) 
and possibly other free variables. Then
\[ 
\forallstfin z\exists x\forall y \in z B(x,y) 
\leftrightarrow
\exists x\forallst y B(x,y).
\tag{I}
\]
Finally, let 
\(C(z)\) 
be a formula, internal or external, with free variable z and possibly other free variables. Then
\[ 
\forallst x\existsst y\forallst z (z \in y \leftrightarrow z \in x \wedge C(z)).
\tag{S}
\]
\citep[p.1166]{Nelson1977} 
\end{quotation}

We may not use external predicates to define subsets,
as the axioms of ZFC that would have to be used to prove the existence of these
subsets do not know how to interact with the undefined predicate ‘\st’.
(It is for this reason that not all of Robinson’s nonstandard analysis
can be reconstructed in IST.)
It is the role of the
standardization axiom (S)
to form standard subsets of standard sets.
Note that standard sets may well contain nonstandard elements;
we will see that
\(\numnat\)
does,
and in a sense that is the whole point of IST.

As IST is not an ontological enrichment of 
\(\numreal\),
Weyl’s question
how the dimensions of the finite and the infinitesimal
are related
for IST does not point to a problem with an ontological aspect.

William Powell proved,
by model-theoretical means, 
that IST is conservative over ZFC
and hence consistent
relative to ZFC
\citep[section 8]{Nelson1977}.%
\footnote{Nelson himself later provided a purely syntactical proof 
that proofs in IST can be reduced to proofs in
a standard system ZFC[V]
which is itself conservative over ZFC 
\citep{Nelson1988}.
Kanovei and Reeken have shown that actually the presentation of IST
there is stronger than that in \citealt{Nelson1977},
and that not all properties of the later version are shared by the earlier one.
However,
they add that this is the case if only bounded sets are considered,
and that in practice these are the ones that matter
\citep[p.128]{Kanovei.Reeken2004}.}

In the following,
we will look at a few theorems of IST,
in order to demonstrate how IST proves that
\(\numnat\)
contains nonstandard numbers,
which are greater than any standard number;
this mostly with an eye on the discussion
of Reeb’s Brouwerian take on IST
in section~\ref{SectionReeb}.%
\footnote{The references for theorems~\ref{firstthmNelson}–\ref{lastthmNelson}
are \citealt{Nelson1977}  and \citealt{Nelson2002}.}
Their reciprocals are infinitesimals.%
\footnote{In IST it is also possible to prove of
\(\numreal\)
directly that it contains infinitesimals.
But the approach through
\(\numnat\)
fits Reeb’s motivation better.}

First,
a strengthening of 
(T) together with its dual form:

\begin{theorem}[T\(^s\)]%
\label{firstthmNelson}\(\forallst t_1\dots \forallst t_k(\forallst x A(x, t_1,\dots, t_k)
\leftrightarrow
\forall x A(x, t_1,\dots, t_k))\) 
\end{theorem}

\begin{proof}
From  (T) and \(\forall x \phi(x) \rightarrow \forallst x \phi(x)\).
\end{proof}

\begin{theorem}[T\(^{sd}\)]
\(\forallst t_1\dots \forallst t_k(\existsst x A(x, t_1,\dots, t_k)
\leftrightarrow
\exists x A(x, t_1,\dots, t_k))\) 
\end{theorem}

\begin{proof}
Apply 
T\(^s\)
to
\(\neg A\),
negate both sides of the bi-implication,
and use
\(\neg\forall x\neg A 
\leftrightarrow
\exists x A\).
\end{proof}

An immediate consequence of 
T\(^{sd}\) is:

\begin{theorem}\label{thm-uniqueness}
If 
\(\exists !x A(x)\),
and 
\(A(x)\)
is internal,
then
\(x\)
is standard.
\end{theorem}

In particular,
\(\numnat\)
is standard.

There is a dual of (I):

\begin{theorem}[I\(^d\)]
\(\existsstfin z\forall x\exists y\in z B(x,y) 
\leftrightarrow
\forall x\existsst y B(x,y)\)
\end{theorem}

\begin{proof}
Apply 
(I)
to 
\(\neg B(x,y)\),
negate both sides,
push the negations inward,
and cancel double negations.
\end{proof}

The key theorem of IST is this,
which entails that nonstandard objects exist  formally:

\begin{theorem}\label{IST-theorem}
Let
\(X\)
be a set.
Then 
\[
X
\text{ is a standard finite set}
\Leftrightarrow
\text{Every element of 
\(X\)
is standard}
\]
\end{theorem}

\begin{corollary}
Every infinite set has a non-standard element.
In fact,
it has infinitely many non-standard elements,
because whenever we remove one non-standard element from it, 
the theorem applies again.
\end{corollary}

First we prove 

\begin{lemma}\label{IST-lemma}
\[
X
\text{ is a subset of a standard finite set}
\Leftrightarrow
\text{Every element of 
\(X\)
is standard}
\]
\end{lemma}

\begin{proof}
Set
\(B(x,y) = x \in X \wedge x \neq y\)
and apply the dual of idealisation (I\(^d\)) to
\(\neg B(x,y)\):
\[
\existsstfin z\forall x\exists y\in z \neg(x \in X \wedge x \neq y)
\leftrightarrow
\forall x\existsst y \neg(x \in X \wedge x \neq y)
\]
Applying logic to the the right hand side,
we get
\[
\existsstfin z\forall x\exists y\in z \neg(x \in X \wedge x \neq y)
\leftrightarrow
\forall x \in X(x \text{ standard})
\]
and to the left hand side,
\[
\existsstfin z(X \subseteq z) 
\leftrightarrow
\forall x \in X(x \text{ standard})
\]
\end{proof}

\begin{proof}[Proof of Theorem~\ref{IST-theorem}]
Left to right:
The assumption that
\(X\)
 is a standard finite set
gives,
together with
\(X \subseteq X\),
\(\existsstfin z(X \subseteq z)\),
and now apply Lemma~\ref{IST-lemma} from left to right.

Right to left:
Assume that every element of 
\(X\) 
is standard.
By Lemma~\ref{IST-lemma},
from right to left,
we have 
\(\existsstfin z(X \subseteq z)\).
By ZFC,
the
power set of 
\(z\),
\(P(z)\),
exists:
\[
\exists x\forall y(y \in x \leftrightarrow y \subseteq z)
\]
This formula is internal,
and 
\(P(z)\)
is unique,
so, 
by
Theorem~\ref{thm-uniqueness},
\(P(z)\)
is standard.
It is also finite,
because 
\(z\) is. 
Applying the proof in the previous paragraph for the direction from left to right
to
\(P(z)\),
all elements of 
\(P(z)\)
are standard,
and as
\(X \in P(z)\),
in particular 
\(X\)
is.
Finally,
\(X\)
is finite because 
\(z\)
is.
\end{proof}

By the corollary to
Theorem~\ref{IST-theorem},
\(\numnat\)
contains a nonstandard number;
this may be thought of as a
proof that the natural numbers we usually work with,
\(0, 1, 2, \dots\)
do not exhaust 
\(\numnat\).
This idea became important to Reeb; see below, section~\ref{SectionReeb}.
Also by the corollary,
there is no set containing exactly those natural numbers that are standard natural numbers;
as mentioned,
the set-forming principles of ZFC
do not have a grip on
the predicate 
‘\st’.
Finally,
we have
\begin{theorem}\label{theorem-greater}\label{lastthmNelson}
The nonstandard numbers in
\(\numnat\)
are greater than
all its standard elements.
\end{theorem}

\noindent
First we prove

\begin{lemma}\label{lemma-standardequal}
Two standard sets are equal if they have the same standard elements.
\end{lemma}

\begin{proof}
Apply (T)  to
\(A(x, t_1, t_2) = x \in t_1 \leftrightarrow x \in t_2\).
\end{proof}

\begin{proof}[Proof of Theorem~\ref{theorem-greater}]
Let 
\(n \in \numnat\)
be nonstandard.
By
(S),
there exists a
standard subset of 
\(\numnat\),
notation
\(^{S}\{z \in \numnat \mid z < n\}\),
such that it includes
all standard elements
of
\(\numnat\)
that
satisfy
\(z < n\);
by Lemma~\ref{lemma-standardequal},
that set is unique.%
\footnote{But (S) does not guarantee that it does not also contain nonstandard elements.}
Obviously,
any standard element of
\(^{S}\{z \in \numnat \mid z < n\}\)
is a standard element of
\(\numnat\).
On the other hand,
if 
\(z\)
is a standard element of
\(\numnat\),
then the set
\(\{w \in \numnat \mid w \leq z\}\)
is a standard finite set.
Theorem~\ref{IST-theorem}
entails that all its elements are standard,
hence
\(n \not\in \{w \in \numnat \mid w \leq z\}\),
and
\(z < n\).
It follows that 
\(^{S}\{z \in \numnat \mid z < n\}\)
and
\(\numnat\)
have the same standard elements.
Both are standard sets,
so Lemma~\ref{lemma-standardequal} applies and
\(^{S}\{z \in \numnat \mid z < n\} = \numnat\).
\end{proof}

From a 
radically formalist position
the axioms could be left unmotivated,
once the axioms are shown or at least believed to be consistent.
Nelson’s paper includes a 
(relative) 
consistency proof;
in his later book chapter,
there are informal considerations
for accepting them.
I single out the one for 
(I),
as it is the one that formally implies the existence
of nonstandard objects:
\begin{quote}
The intuition behind (I) is that we can only fix a finite number of objects at a time. 
To say that there is a 
\(y\) 
such that for all 
fixed 
\(x\) 
we have 
\(A\) [i.e., \(B(x,y)\)]
is the same as saying that for any fixed finite set of 
\(x\)’s 
there is a 
\(y\)  
such that 
\(A\) 
holds for all of them.
\citep[p.5]{Nelson2002}
\end{quote}

Nelson acknowledged of course that there is informal discourse in mathematics,
and the term ‘fixed’ belongs to that realm 
\citep[p.1]{Nelson2002}.
But since Nelson’s reflected judgement is that there is no mathematical reality,
be it intuitionistic or Platonic,
and that strictly speaking mathematics consists in formal systems,%
\footnote{See on this also chapter 32,
‘A modified Hilbert Program’,
in
\citealt{Nelson1986}.}
the strict counterpart of the informal discourse’s notion of being fixed
for him must be found in a property of formal proofs.
The statement
‘we can only fix a finite number of objects at a time’
is then mapped
to the fact that each proof in the formal system at hand
is a finite object,
which therefore leaves room for only finitely many
occasions to
define
(fix)
individual objects
and prove or assume their existence.

An analogous argument for a simpler case is this (presentation after \citealt[p.1195]{Palmgren1993}):

\begin{theorem}
Extend Peano Arithmetic with a
constant
\(ω\)
and the
axiom schema
\(ω > n\),
to obtain a nonstandard theory PA\(^*\).
Then 
PA\(^*\) 
is conservative over PA.
\end{theorem}

\begin{proof}
Assume we have a formal proof of
\(A(ω)\).
As the formal proof is finite,
only finitely many instances of the schema
occur in it,
\(ω > n_1,\dots,ω > n_k\).
Define
\(m=\text{max}(n_1,\dots,n_k)+1\),
and
replace,
in the original proof,
\(ω\)
by
\(m\)
everywhere.
\end{proof}

Below,
in section~\ref{SectionReeb},
we will see that 
according to Reeb,
who was not a formalist and who held that
there is a mathematical reality
which furthermore is constructive,
there is  in mathematical reality a
motivation for introducing the predicate
‘\st’
in ZFC,
and from there for accepting the formal theory IST.

\section{Three Brouwerian desiderata}\label{SectionBrouwerCriteria}

Turning now to Brouwer’s writings,
one may distill three desiderata for
constructions of
infinitesimals:
\begin{enumerate}

\item They should be intuitionistic constructions, 
i.e., be built up starting from 
‘the basic intuition of mathematics’:
\begin{quote}
the substratum of all perception of change, 
which is divested of all quality, 
a unity of continuous and discrete, 
a possibility of the thinking together of several units, 
connected by a ‘between’, 
which never exhausts itself by the interpolation of new units.
(\citealt[p.8]{Brouwer1907}, trl.~\citealt[p.17]{Brouwer1975}):
\end{quote}
Further on in the dissertation
Brouwer specifies that this basic intuition
consists in the awareness 
of time as pure change 
(\citealt[pp.98–99]{Brouwer1907})
that 
it and our construction acts on it
are not of a linguistic nature  
(\citealt[p.169]{Brouwer1907})
and that there is not also
a spatial continuum that is a priori given to us
(\citealt[p.121]{Brouwer1907}).

\item Logical reasoning about them should be done according to the
nature of mental
constructions.
This respects the essential non-linguistic character of
mathematical
construction,
and the nature of logic,
such as Brouwer describes it,
as a
study of the patterns in descriptions of that activity 
(\citealt[pp.131-132]{Brouwer1907};
\citealt{Brouwer1908C}).
Whatever logical principle
one has recognised as correct on this conception should be
allowed in one’s
reasoning.

\item They should be geometrical in nature. 
Brouwer defines geometry as 
follows:
\begin{quotation}
Geometry is concerned with the properties of spaces of one
or more dimensions.
In particular it investigates and classifies sets,
transformations and transformation groups in these spaces.

The spaces under consideration are built up out of one or
more Cartesian simplices,%
\footnote{In two dimensions a simplex is a triangle with all its interior points;
in three dimensions a pyramid with a triangle as its base.}
which can be connected in different ways;
consequently a space is not completely defined by its
dimension alone.
(\citealt[p.15]{Brouwer1909A}, trl.~\citealt[p.116]{Brouwer1975})
\end{quotation}

\end{enumerate}
The classical and constructive nonstandard-models
of the previous section obviously do not meet these desiderata,
and neither does a purely axiomatic approach.
But it is of course just as clear that there will be no direct
intuitive construction of infinitesimals on the one-dimensional continuum.
If in the next section the reason for this is spelled out,
that is because it adds relief to Brouwer’s construction
in section~\ref{SectionBrouwerRealnumber}
of a real number that is greater than
\(0\),
but of which we cannot indicate a positive distance from 
\(0\);
this is the kind of construction that Vesley took up,
as we will see in section~\ref{SectionVesley}.

At the same time,
it should also be noted that
even in classical nonstandard analysis 
there is a large constructive element om the following sense:
Once non-constructive methods have been employed to obtain infinitesimals,
the reasoning often proceeds constructively,
employing standardisation at the very end to return
to the realm of the standard.
That topic has recently been explored in great detail
in
\citealt{Sanders2018}.

\section{The impossibility of a direct construction on the one-dimensional continuum}

The intuitive continuum
as given in 
what Brouwer calls
the basic intuition of mathematics
has no scale on it.%
\footnote{The primary reference for this paragraph and the next is Brouwer’s dissertation,
\citealt[pp.8–11]{Brouwer1907}, 
but the argument is general.}
Brouwer’s ‘between’
is not intrinsically tied to intervals of
any size,
because if there is no scale then there are no sizes,
in particular no infinitesimal ones.
This intuitive ‘between’
precedes the construction of a scale,
and the scale is constructed
by 
‘the interpolation of new units’
\emph{on}
it.

Putting a scale on the intuitive continuum is itself
a construction process that takes place over time.
The human mind is limited in such a way that
in a given time interval
we can only place finitely many points of a scale
on the intuitive continuum,
or begin a potentially infinite sequence of such
placements.
Between any two previously placed points,
an intuitive continuum remains,
and
if we choose to do so,  
we can place a further point
on this ‘between’,
and thereby continue
the construction of our scale
into it.
If we iterate this everywhere,
in a potentially infinite process,
we construct
a countable,
everywhere dense scale.
We can correlate the points on this scale
with
any number system we have constructed of the order type
of 
$\numrat$; 
we may begin by correlating an arbitrary point 
on the scale with $0$ and another one with $1$.
We thus obtain 
a ‘measurable continuum’ 
(\citealt[p.11]{Brouwer1907}).

Once the construction process of the rational scale has begun,
we then construct 
points
or real numbers
(including the embedding of the rationals)
$p$
by constructing
potentially infinite
sequences
of nested
intervals 
with endpoints on the scale
$p_0, p_1, p_2, \dots$
As a point does not exist on the continuum prior to our
construction,
it is identified with the developing 
sequence,
as opposed to an independently existing limit to which the sequence
converges.
If two points 
$p$
and
$q$
are not equal,
this unequality must consist in the fact that
starting from a finite index $n$,
the intervals
$p_n$
and
$q_n$
do not overlap.
As the intervals are determined by rationals,
we can determine a natural number 
$m$ such that 
$mp > q$
or,
as the case may be,
$mq > p$.
So the system is Archimedean,
and infinitesimals or
intervals of infinitesimal length
do not exist.
To construct an infinitesimal interval on the intuitive continuum,
we would have to be able to construct a point
$p$
such that $\neg(p = 0)$
but of which it is contradictory to assume that the
unequality to $0$ arises
at some 
$p_n$
for natural $n$
(i.e., at a rational distance from $0$).
That is impossible.

We will see in section~\ref{SectionBrouwerRealnumber} 
that,
with the admission into intuitionism of choice sequences,
we 
\emph{can}
construct a real number $r$ that is
unequal to $0$ and of which we cannot indicate
the interval
$r_n$
at which the 
unequality to
$0$
arises
\emph{until a certain proposition
$P$
has been decided}.
But before
$P$
has been decided we can already show
that it is contradictory to suppose that the 
unequality to
$0$
arises at 
\emph{no}
finitely indexed interval.

The fact that there are,
in the basic intuition of mathematics, 
no direct motivation and no direct construction
for infinitesimals
(as objects constructed on the one-dimensional continuum),
the development of a theory would have to proceed,
just as in the classical case,
either by an embedding of the standard real numbers into a more-dimensional structure,
and thereby no longer take propositions in analysis of the one-dimensional continuum
at face value,
or construe talk about nonstandard objects as talk about certain standard objects.

\section{Brouwer’s non-Archimedean numbers}

Early Brouwer’s  construction of non-Archimedean numbers
discussed in this section was not meant to  lead up to
a form of infinitesimal analysis.
However,
that 
would have been a first step,
so the general considerations are of interest to the present discussion.

When Brouwer was working on his dissertation,
non-Archimedean fields and geometries  
had been constructed by
notably Veronese,
Levi-Civita,
Pasch,
Hilbert,
and Vahlen,%
\footnote{An extensive historical treatment
is \citealt{Ehrlich2006}.}
and these he refers to in his notebooks and in his thesis,
with an emphasis on Hilbert.
Brouwer criticised these approaches:
Veronese’s was not constructive in his sense,
and those of Pasch, Hilbert and Vahlen
are not geometrical in his sense.

In Veronese’s 
\emph{Fondamenti di Geometria} 
of 1891,
a real number is construed as the ratio of two magnitudes (both of the same species),
one of which is designated to be the unit;
and the existence of a segment 
that is infinitesimal with respect to another
is 
\textit{postulated}.
Veronese can do so because,
as he states in his introduction,
‘A thing postulated by thought
one can consider as given to thought,
and inversely’ 
\citep[Introduzione, section 18, trl.~mine]{Veronese1891}.
For Brouwer,
on the other hand,
only that what has been constructed from the basic
intuition qualifies as given.
In a notebook that predates his dissertation,
he comments:
\begin{quote}\label{BrouwerVeronese}
Veronese’s fuss,
with his constantly introducing hypotheses,
is nothing but
forming logical assemblies;
if for certain things
(I do not know whether they exist)
such and such relations hold,
then also
such and such relations.
\citep[Notebook 3, p.35, trl.~mine]{Brouwer19041907}
\end{quote}

Hilbert’s non-Archimedean geometries are criticised
for their
non-geometrical nature.
In the synopsis for his dissertation,
Brouwer writes:
\begin{quote}
Hilbert’s pseudo-geometries are
(in contrast to the non-Euclidian)
of little importance,
because they have been built within
a rather ‘far-fetched’ building [i.e., construction]
(while the non-Euclidian in the ordinary Cartesian space).
(\citealt[p.405, trl.~mine]{Brouwer19041907})
\end{quote}
Brouwer is referring to the fact that
the coordinates of points in the space are not real numbers
but objects of a higher type,
namely certain algebraic functions on the real numbers
\citep[p.25]{Hilbert1899}.
Such algebraic functions may themselves be represented 
(extensionally)
geometrically,
but Brouwer’s hesitation here would be that 
each such representation is not
a point in a
(\(n\)-dimensional)
Cartesian space.
Hilbert,
of course,
proposed 
his non-Archimedean geometry 
first of all 
in the service
of an independence proof of the Archimedean axiom,
and then Brouwer’s considerations are not that important.
But such geometries soon turned out to be of interest in their own right.

In his dissertation,
Brouwer presents
an
alternative
non-Archimedean continuum,
a(n intended)
construction
in his specific, non-axiomatic sense of that term,
in an ordinary
Cartesian space 
(\citealt[pp.67–73]{Brouwer1907}).
It
is,
he states
(\citealt[p.140n]{Brouwer1907})
a generalization
of that in
\citealt[section 34]{Hilbert1899}.

By Brouwer’s de facto phenomenological consideration,
given in the previous section,
a non-Archimedean (mathematical) continuum
cannot be constructed on 
the one-dimensional intuitive one. 
His strategy therefore
was to construct a multi-dimensional mathematical continuum
and define a subset on it 
which he calls
‘the pseudo-continuum’.
The pseudo-continuum can be 
linearly ordered
and the 
one-dimensional continuum
embedded into it.

Brouwer starts with
an infinite-dimensional Cartesian space
of
\((ω^{*}+ω)^n\)
dimensions,
where 
\(ω^{*}\)
is
\(\ldots, -3, -2, -1\).
Each coordinate has 
(instead of a letter)
an ordinal number
in between
\(-ω^n\)
and
\(ω^n\),
which can be written in the form
\(a_1ω^{n-1}+\cdots+a_{n-1}ω+a_n\),
with
\(-ω < a_i < ω\).
To the coordinate then is associated
the 
\(n\)-tuple of indices
\(\langle a_1,\ldots,a_n \rangle\).

The ‘pseudo-continuum’ 
now consists of the subset of points in the
space with the property that 
for all their coordinates whose value is not 
\(0\)
we can indicate lower bounds on the 
\(a_i\):
the property,
in other words,
that
non-zero values are not found at
arbitrarily low coordinate numbers.
This in turn means
that the coordinate numbers
corresponding to those 
\(n\)-tuples
form a well-ordered set
(i.e., a set of which each non-empty subset has a first element).
So 
for any two distinct points
\(p,q\)
there will be
a smallest coordinate number at which they differ,
and therefore the set can be linearly ordered.

The one-dimensional
Archimedean continuum
is embedded into the pseudo-continuum
by assigning the point
on the former
with coordinate
\(x\)
to the point on the latter
whose coordinates are all 
\(0\)
except that its
\(0\)-coordinate is
\(x\).
One may view the pseudo-continuum
as a real continuum
with infinitely many points
inserted to the right and left of each
real point,
and with infinitely many pseudo points
to the left and right of the real continuum as a whole.

The operations 
\(+\)
and
\(\times\)
are understood group-theoretically,
that is,
as parametrised transformation operations
\(+ a\)
and
\(\times a\).
\begin{theorem}[Brouwer 1907]\label{grouptheorem}
On a measurable
continuum,
there is only one construction for
a two-parameter continuous uniform
transformation group
\[x^\prime = c_1\times x + c_2\]
namely the one
in which
\(+\)
and
\(\times\)
are ordinary addition and multiplication
(and hence
commutative).
(\citealt[pp.32–33]{Brouwer1907})
\end{theorem}

On the pseudo-continuum
he then defines
a 
two-parameter continuous uniform
transformation group
that preserves
\(+\)
and
\(\times\)
on the embedded one-dimensional continuum,
but whose multiplication
is not commutative on the
pseudo-continuum as a whole.

The group operation
\(+\)
on the pseudo-continuum
is induced by the 
\(+\)
operation on the scale of each of the coordinates.
It is associative and commutative.
The operation
\(\times\)
on the pseudo-continuum
is defined as an operation that shifts coordinates;
\(1_1\times\)
shifts the number of a coordinate to the right by one 
while mapping the 1-points of the scales of each coordinate onto one another.
Likewise,
\(1_ω\times\)
shifts the number of a coordinate 
to the right
by 
\(ω\) 
while mapping the 1-points onto one another.
Brouwer determines conditions on the
the choices of the 1-points
on the scales of coordinates
\(1, ω, ω^2, \ldots\)
that will guarantee that 
\(\times\)
is associative and distributive with
\(+\).

But \(\times\) on the pseudo-continuum
has been defined so as not to be
commutative:
for example,
\[
1_1 \times 1_ω = 1_{ω+1}
\]
but
\[
1_ω \times 1_1 = p_{ω+1}
\]
where 
\(p\)
is the point on the scale of coordinate 
\(1\)
chosen to be the 
\(1\)-point on that scale.
In general,
\(1_{ω+1}\)
and
\(p_{ω+1}\)
are not equal.


Theorem~\ref{grouptheorem} then implies that 
this pseudo-continuum is not a measurable one.
Brouwer remarks that this pseudo-continuum is
not continuous in Dedekind’s sense
(which would have implied it is Archimedean),
but it is in Veronese’s
(\citealt[pp.72–73 ]{Brouwer1907}).%
\footnote{See on this point 
\citealt[69-71]{Ehrlich2006}.}

An objection to the way the pseudo-continuum is constructed
is that it presupposes the Principle of the Excluded Middle:
In order to obtain the linear ordering of its points,
it must be possible to decide whether
a sequence that proceeds infinitely to both sides
is from a certain element onward constant zero to the left.
This is similar to a problem in another part of his thesis,
which is flagged and discussed in the corrections 
that he published in 1917:
When moving down along a branch in a tree,
one cannot,
in general,
decide whether each future node will have
a unique descendant 
(\citealt[p.440]{Brouwer1917A2}).	

At the time,
Brouwer accepted the Principle of the Excluded Middle because he
took 
\(P \vee \neg P\)
to be equivalent to
\(\neg P \rightarrow \neg P\)
\citep[pp.106-107]{Dalen1999}.
Constructively,
it is not;
Brouwer presented the correct reading,
according to which PEM is valid only for decidable propositions,
in \citealt{Brouwer1908C},
‘The unreliability of the logical principles’.%
\footnote{A recent English translation and introduction
is 
\citealt{Atten.Sundholm2017}.}

Moreover,
Brouwer’s construction,
if successful, 
would not be
a field extension of the real numbers,
so it could not have been used to develop a nonstandard analysis.

It is not surprising,
then,
that Brouwer did not develop the theory of this pseudo-continuum any
further.%
The work on non-Archimedean numbers
was superseded by Hahn’s paper
‘\german{Über die nicht\-ar\-chi\-me\-di\-schen Größensysteme}’
\citeyearpar{Hahn1907},
which
appeared just too late to be taken into consideration in 
Brouwer’s thesis,
which was defended on February 19 of the same year.
But 
in 1917,
Brouwer referred to it in his list of additions and 
corrections to his dissertation
specifically for its treatment 
of commutative principal operations
(\citealt[p.441]{Brouwer1917A2}).

One\label{fngoedel} of the few people who seem actually to have studied
Brouwer’s pseudo-continuum is
Kurt Gödel.
In 1941,
by which time he had emigrated to the United States,
he asked his brother
who had remained in Vienna,
to order a copy of Brouwer’s dissertation
for him \citep[pp.190–191]{Atten2015}.
And indeed,
in one of Gödel’s notebooks, 
probably filled in 1942,
one
finds reading notes on Brouwer’s construction.%
\footnote{\citealt{Godel.Papers}, Arbeitsheft 14, pp.21–23;
see its page 14 for the year.}
Gödel at the time was interested
in  the question
if there could be non-human beings 
in whose awareness time is ordered in a  non-Archimedean way
\citep[Max Phil VI (?–July 1942), pp.431–432]{Godel.Papers}.

Brouwer made one last remark on 
non-Archimedean geometry in his second  lecture
in Vienna 1928~– with Hahn,
who had in the meantime become a friend of his, 
in the audience:%
\footnote{Brouwer’s Vienna lectures were invited by a committee of which Hahn was a member  \citep[p.561]{Dalen2005}.}
\begin{quote}
The initial, 
negative attitude towards
these 
[non-Euclidean or non-Archimedean]
geometries
was completely
overcome
by their arithmetisation
due to
Riemann, Beltrami, Cayley,
and,
respectively,
Levi-Civita and Hahn.
In the process,
the peculiar fact came about that
the non-Archimedean continuum,
which
had proved to fulfill the
a priori conditions on the continuum
just as well as the Archimedean,
was brought about in a plausible manner
only with the aid of the latter,
so that the calling into question of the a priori necessity 
of the Archimedean continuum
had to be founded precisely on the a priori consistency
of this continuum.
(\citealt[p.1, trl.~mine]{Brouwer1930A})
\end{quote}

\section{Brouwer: a real number that is greater than \(0\), but not measurably greater}\label{SectionBrouwerRealnumber}
Around 1916,
Brouwer introduced choice sequences
into intuitionistic analysis.%
\footnote{For an introduction to choice sequences,
with particular attention to philosophical and mathematical differences
between Brouwer’s theory and Weyl’s adaptation of it,
see \citealt{Atten.Dalen.Tieszen2002};
for their history, \citealt{Troelstra1982}.}
A choice sequence is
a sequence of natural or rational numbers that
are freely chosen by the Creating Subject,
which is moreover free to
impose restrictions on its choices.
Thus we have
sequences without any restriction on the choices
(lawless sequences)
and
sequences determined by an algorithm or law
(lawlike sequences).
For Brouwer these are the extreme cases,
with many other kinds of choice sequence in between,
notably also choice sequences for which the 
Creating Subject lets its choices depend
on some of its other mathematical activities.
(We will see examples of this latter kind in this section and the next.)

Brouwer’s rationale for reconstructing analysis
in a theory of choice sequences is that this
gives a mathematical,
fully constructive model
of the intuitive continuum
that faithfully mirrors,
not only epistemologically but ontologically, 
the fact that the latter
is not a composition out of discrete elements.

Weyl,
in his intuitionistic period,
accepted the theory of choice sequences in a modified form that,
however,
made it incoherent
\citep[section 3]{Atten.Dalen.Tieszen2002}.
Be that as it may,
of some interest for our present theme is Weyl’s 
intention to accept universal quantification over 
lawless sequences but to insist that instantiations
are lawlike,
for in this way he is in effect treating
lawless sequences
as 
nonstandard objects.
For Weyl,
only lawlike
choice sequences could exist as individuals.
Brouwer’s particular choice sequence
that is the topic of the present section
would not have been acknowledged 
as
an individual mathematical object
by Weyl.

The introduction of choice sequences
did not affect early Brouwer’s observation on the
impossibility to construct a non-Archimedean scale.
The reason is that the latter observation
is made at such a high level of generality that
it also subsumes
choice sequences.
Yet,
in the 
\textit{Cambridge Lectures} 
(1946–1951)
Brouwer states that
\begin{quote}
the intuitive ‘between’ surely requires 
as well that the continuum contains further point cores
between,
for instance,
the origin on the one hand and all rational point cores on the other.
\citep[p.50]{Brouwer1981A}
\end{quote}
A ‘point core’ is an equivalence class (or rather a ’species’)
of choice sequences,
the criterion being that they are all
co-convergent.
Brouwer seems to be saying here that one can construct points  that are not 0 yet
whose distance from 0 is smaller than any rational number we will ever construct;
he seems to be saying that we can construct infinitesimals.
It will turn out that this is not quite what is meant.%
\footnote{The same construction is also in \citealt{Brouwer1948A},
but there Brouwer does not add the comment quoted above.}

To show Brouwer’s argument for this claim,
his definitions of some order relations are needed
\citep[p.1246n]{Brouwer1949C}.
Let
$β$ 
and 
$γ$
be two real numbers,
i.e., two convergent infinite sequences of rational numbers
$β(n)$
and
$γ(n)$.
Define
\(β\klo γ\),
‘\(β\) is measurably smaller than \(γ\)’
as
\[
\exists m,n\in\numnat\forall v\in\numnat\left(v \geq m \rightarrow γ(v)-β(v) > \frac{1}{2^n}\right)
\]
Correspondingly,
\(γ \gro β\)
means that
\(γ\) 
is 
‘measurably greater’
than
\(β\)
\citep[p.3]{Brouwer1951}. 
Write
\begin{align*}
& β \neq γ \quad \text{ for } \quad\neg(β=γ)\\
& β \geq γ \quad\text{ for }  \quad\neg(β \klo γ)\\
& β > γ \quad\text{ for } \quad β \geq γ \wedge β \neq γ
\end{align*}
So being measurably greater,
defined as a double existential statement, 
is a positive property,
while being greater,
defined as in effect a conjunction of two negations, 
is a negative property.

The apartness relation 
\citep[p.1246n]{Brouwer1949C} 
is defined as follows: 
\[
β \apart γ \equiv \exists k\in\numnat\left(|β-γ|>\frac{1}{2^k}\right)
\]
or,
equivalently,
\[
β \apart γ \equiv β \klo γ \vee β \gro γ
\]

The definition of a choice sequence,
and so in particular of a point core that it represents, 
may be made
to depend on what goes on in the subject’s other activities
in between the choices of the elements in this sequence,
notably with respect to attempts to settle a certain
problem.
For example,
in between two choices,
the subject may have
decided a proposition
\(P\),
or have
tested it.
A proposition 
\(P\)
is decided
by either proving
\(P\)
or
proving
\(\neg P\);
it is tested
by either proving
\(\neg P\)
or
\(\neg\neg P\).
Decidability implies testability.
If 
\(P\) 
holds,
then so does
\(\neg\neg  P\),
and if 
\(\neg P\),
then 
\(\neg P\);
so
\(P \vee \neg P\)
implies
\(\neg P \vee \neg\neg  P\).
But testability does not imply decidability.
For suppose we
can
prove
\(\neg\neg  P\)
but not 
(yet) 
\(P\);
then
\(P\) 
has been tested
but is still undecided.

\begin{wce}[Brouwer]
\label{Brouwerweakcounterexample}
There is no hope of showing that
$\forall α(α > 0 \rightarrow α \gro 0)$.%
\footnote{In the next section,
we will see that Brouwer also had a proof of the actual negation,
\(\neg\forall α(α > 0 \rightarrow α \gro 0)\) (Theorem~\ref{Brouwerstrongcounterexample}).}
\end{wce}

\begin{pargmt}
Let
$P$
be
a proposition
that
we cannot test,
in the 
weak sense that we do not now possess a proof of
$\neg P \vee \neg\neg  P$.%
\footnote{Brouwer could have given this argument in terms
of an undecidable proposition instead
of an untestable one.
The reason he uses an untestable
one is that in his paper he exploits almost the same construction
to prove that
$\neq$
cannot be defined as a disjunction of
$<$
and
$>$,
as that would lead to the contradiction that an untestable proposition is testable.
For further discussion of Brouwer’s weak and strong counterexamples,
see 
\citealt{Atten-forthcoming}.}

The Creating Subject constructs a real number $α$ in a
choice
sequence of rational numbers $α(n)$, as follows:
\begin{itemize}
\item As long as,
when making the choice of $α(n)$,
the Creating Subject
has obtained evidence neither of $P$
nor of $\neg  P$,
$α(n)$ is chosen to be 0.

\item If between the choice of $α(n-1)$ and
$α(n)$,
the Creating Subject
has obtained evidence
of $P$,
or has obtained evidence
of $\neg  P$,
$α(n)$ and all $α(n+k)$ are
chosen to be
$(\frac{1}{2})^{n}$.
\end{itemize}
The choice sequence $α$ converges,
hence α is a real number.%
\footnote{Let 
\(ε\)
be given,
and determine an
\(n\)
such that
\(2^{-n} < ε\).
Construct the sequence
\(α\)
up to
\(α(n)\),
which can be done as each choice is decidable.
If
\(α(n)=0\),
all further choices will be in the interval
\([0,2^{-(n+1)}]\)
and hence within
\(ε\)
from one another.
If
\(α(n) \neq 0\),
then the choices in
\(α\)
have already been fixed,
and hence within
\(ε\)
from one another.}
We have
\[
α = 0 \leftrightarrow \neg P \wedge \neg\neg P
\]
Hence $α \neq 0$. 
We also have
$\neg(α \klo 0)$
because, 
by definition of $α$,
no
$α(n)$ is ever smaller than 0;
and their conjunction gives
$α > 0$.

But we do not have the stronger
$α \gro 0$
because
if we had,
then
\[
\exists m,n\in\numnat\forall v\in\numnat\left(v \geq m \rightarrow α(v) > \frac{1}{2^n}\right)
\]
and this is only possible if
$P$ 
would have been decided,
and hence tested;
but this contradicts
the hypothesis that
$P$ 
cannot be tested yet.
\end{pargmt}

This is what Brouwer means when,
in the quotation from the 
\textit{Cambridge Lectures}
above, 
he says that
there are point cores between
the origin and all positive rational point cores.
If,
by developing more mathematics,
we do come in a position to test
$P$,
that is we
can find a proof of 
$\neg P$
or a proof of
$\neg\neg  P$,
then the number 
$α$
becomes rational,
and,
in the sense of these order relations,
no longer lies 
between
0 and all rational point cores.
Note that
this does not mean that
$α$
was irrational before.%
\footnote{Before,
it was a growing construction for a real number
that had  yet acquired neither the property
of being rational,
nor that of being irrational.}

Brouwer does not go on to connect
his example of 
a number between 0 and all the rationals
in any way
to infinitesimals.
That would be done by 
Vesley.

\section{Vesley’s \(α\)-infinitesimals}\label{SectionVesley}
Vesley realised that 
real numbers like the one in Brouwer’s example
are,
although not infinitesimals in an ontological sense, 
in an important respect similar to infinitesimals
\citep{Vesley1981}.%
\footnote{I don’t think Vesley knew of that particular passage in Brouwer,
which was published only in 1981,
but
he was of course very familiar with this kind of reasoning,
e.g.~\citealt{Kleene.Vesley1965}.}
The same observation was made independently
in
\citealt[p.191]{Dalen1988}.

Vesley appeals to 
Kripke’s Schema:
\begin{equation}
\exists α(\exists n\, α(n)=1 \leftrightarrow
P)
\tag{KS}
\end{equation}
where 
\(P\)
is a variable for propositions,
and 
\(α\)
for choice sequences.
Brouwer had demonstrated this 
before Kripke did
but never used it again,
and instead reasoned from the general principles 
from which KS quickly follows.%
\footnote{See \citealt[p.4]{Brouwer1954A}
for Brouwer’s demonstration,
and
\citealt[p.295]{Myhill1966}
for the observation that this is KS.
Brouwer does not literally state KS;
he constructs,
from an arbitrary proposition
$P$
that as yet cannot be tested,
an infinite sequence
$C(γ,P)$,
and shows that 
truth of 
$P$
and 
rationality of 
$C(γ,P)$
are equivalent.
However,
the construction
of a witness for KS from
$C(γ,P)$
is immediate;
and Brouwer's reasoning towards the existence
of 
$C(γ,P)$
goes through for any
$P$,
not only untestable ones.}
These principles were later codified by Kreisel
in the so-called 
‘Theory of the Creative Subject’ (or ‘Creating Subject’).
For discussion
of the Creating Subject and KS,
see
\citealt{Myhill1966};
\citealt{Kreisel1967b};
\citealt[ch.4]{Troelstra.Dalen1988};
\citealt[ch.5]{Atten2004};
and
\citealt{Atten-forthcoming}.

I will adapt Vesley’s construction somewhat  to Brouwer’s way.
Let
\(α\)
be a choice sequence
(whether of convergent rationals or not).
Define
a real number
\(x\)
as follows:
\begin{itemize}
\item As long as,
when making the choice of 
\(x(n)\),
the Creating Subject
has obtained evidence of neither 
\(\forall n α(n)=0\)
nor of
\(\neg\forall n α(n)=0\),
\(x(n)\) 
is chosen to be 0.

\item If between the choice of 
\(x(n-1)\)and
\(x(n)\),
the Creating Subject
has obtained evidence
of 
\(\forall n α(n)=0\),
\(x(n)\) and all 
\(x(n+k)\)
are
chosen to be
\((\frac{1}{2})^{n}\).

\item If between the choice of 
\(x(n-1)\) 
and
\(x(n)\),
the Creating Subject
has obtained evidence
of 
\(\neg\forall n α(n)=0\),
\(x(n)\) 
and all 
\(x(n+k)\) 
are
chosen to be
\(-(\frac{1}{2})^{n}\).
\end{itemize}

Then we have
\[
\exists x \in \numreal[(x \gro 0 \leftrightarrow \forall n α(n)=0) \wedge (x \klo 0 \leftrightarrow \neg\forall n α(n)=0)]
\]
and,
since 
\(α\)
was an arbitrary choice sequence,
\[
\forall α\exists x\in \numreal[(x \gro 0 \leftrightarrow \forall n α(n)=0) \wedge (x \klo 0 \leftrightarrow \neg\forall n α(n)=0)]
\]

Define for every choice sequence
the species of real numbers
\(L(α)\)
and
\(M(α)\):
\begin{align*}
x  \in L(α) & \equiv [x \apart 0 \leftrightarrow \forall n α(n)=0 \vee \neg\forall n α(n)=0]\\
x \in M(α) & \equiv \exists y \in L(α) \neg(\lvert x\rvert \gro \lvert y\rvert)
\end{align*}
Vesley points out that  one can then prove:

\begin{theorem}
\(M(α)\) is a subring of the intuitionistic \(\numreal\).
\end{theorem}

\begin{theorem}
\(\neg\forall α \exists x\in M(α) \exists n (n \cdot x \gro 1)\)
\end{theorem}

The species 
\(M(α)\)
is called that of the
α-infinitesimals.
These are
elements of 
(intuitionistic)
\(\numreal\),
and in this sense the conception is a little like
that in Nelson’s IST,
where 
the infinitesimals are
elements of
classical \(\numreal\).

Vesley observes 
that,
although we want to think of the α-infinitesimals as very small,
and for that reason give them this suggestive name,
should the question whether
α is 
\(0\)
everywhere be decided,
\(M(α)\)
becomes
\(\numreal\).
This
goes against the very idea of
an infinitesimal.
On the one hand,
in light of open-endedness of mathematics,
there will always be new open problems,
so this is not much of an objection
against the existence of α-infinitesimals in general.
On the other hand,
this also means that
these infinitesimals only behave as infinitesimals under universal quantification.
We cannot prove that all
species of α-infinitesimals are non-Archimedean;
only that it is not the case that none of them is.
Instead of
\[
\forall α\neg\exists x\in M(α)\exists n (n\cdot x \gro 1)
\]
we only have
\[
\neg\forall α\exists x\in M(α)\exists n (n\cdot x \gro 1)
\]
And that the latter cannot be strengthened to the former
is intrinsic to the whole construction.
An individual 
$M(α)$
will be non-Archimedean
as long as it is undecided whether 
the values of 
$α$
are 
\(0\) 
everywhere or not,
but becomes Archmedean as soon as this has been decided.
And there cannot be a particular
$α$
for which this is never decided,
for that would imply the existence
of an absolutely undecidable proposition,
which is impossible:%
\footnote{One might think the permanent existence of an
α-infinitesimal 
can be assured 
by starting a sequence starting with 
\(0\)’s
and stipulating that one will 
never make the decision between 
(a)
restricting the remaining choices
to 
\(0\)
and
(b) 
making a choice that is not 
\(0\).
This however will not do,
because by choosing 
\(0\) 
until that
decision is made,
and at the same time
resolving always to postpone that decision,
the result is that the sequence will be constant 0.}

\begin{theorem}[Brouwer, 1907–1908?]
There exist no absolutely undecidable propositions.
\end{theorem}

\begin{proof}
‘Can one ever demonstrate of a proposition, that it can never
be decided? No, because one would have to so by reductio ad
absurdum. So one would have to say: assume that the
proposition has been decided in sense \emph{A}, and from that
deduce a contradiction. But then it would have been proved
that not-\emph{A} is true, and the proposition is decided after
all.’ (Note by Brouwer, as quoted in 
\citealt[p.174 note a]{Dalen2001a}; translation mine)
\end{proof}
That $\neg(P \vee \neg P)$ is contradictory,
and hence that PEM is consistent,
was pointed out in \citealt{Brouwer1908C}.
The quoted argument was never published by Brouwer,
but 
\citealt[p.66]{Wavre1926}
and 
\citealt[p.16]{Heyting1934}
made
the same observation.

Van Dalen,
who as mentioned made the same connection between Brouwer’s
weak counterexample and infinitesimals,
considered it
‘highly unsatisfactory to include subjective phrases such as
“it cannot be shown” in mathematical texts’
\citep[p.191]{Dalen1988},
and points out that
in
\citealt{Brouwer1949A}  
this result is strengthened from the weak counterexample
‘there are real numbers 
\(α\)
that are greater than  
\(0\)
yet cannot be shown to be measurably greater than 
\(0\)’
to the strong counterexample
\begin{theorem}[Brouwer 1949]\label{Brouwerstrongcounterexample}
\(\neg\forall α(α > 0 \rightarrow α \gro 0)\)
\end{theorem}
\noindent
(Note that intuitionistically this does not imply
\(\exists α\neg(α > 0 \rightarrow α \gro 0)\),
which is contradictory.)
Indeed,
in the demonstration of this theorem,
instead of one open problem and unbounded time to solve it,
Brouwer considers the infinity of 
open and solved problems
‘\(α \in \numrat\)’
for all
\(α \in [0,1]\)
with
the added condition,
which arises from the applicability of the fan theorem to 
functions defined on that interval,
that they should all
be solved on the basis of an initial segment of
\(α\)
of uniform length.%
\footnote{Strictly speaking,
Brouwer does not consider the question of rationality of each
\(α \in [0,1]\),
but of each
\(α \in J\),
where 
\(J\)
is a fan that coincides with 
\([0,1]\).
Note also that our notational use of
α is different from that in \citealt{Brouwer1949A}.}
And then not only it cannot be shown that this condition can be met
(weak negation),
it can be shown that it is leads to contradictions if it could
(strong negation).

Vesley observes that in this version of nonstandard analysis,
‘the elegance of classical nonstandard analysis is missing
for the familiar reason that more distinctions must be recognized
intuitionistically’ \citep[p.211]{Vesley1981},
but because of the dependency of his infinitesimals on universal quantification
he also sees a similarity
to  the synthetic differential geometry of Lawvere and Kock,%
\footnote{There,
‘nilpotents’,
which are  
\(δ\) 
such that 
\(δ \neq 0\) 
but 
\(δ^2 = 0\), 
may be cancelled when universally quantified.}
which has been developed much further \citep{Kock2006}.
From a Brouwerian perspective,
Vesley’s approach would philosophically be preferable to 
synthetic differential geometry 
(smooth infinitesimal analysis)
in that the latter involves a postulation 
of the existence of a line segment of infinitesimal length,
which is certainly not given to us in mathematical intuition.%
\footnote{Compare Brouwer’s objection to Veronese’s
postulate above,
p.~\pageref{BrouwerVeronese}.}
Vesley announced a sequel paper to develop the approach further
and to see whether it has advantages of its own.
Unfortunately,
it seems that he gave up on the project.

\section{Reeb: An intuitionistic take on IST}\label{SectionReeb}

An approach to infinitesimals that is
Brouwerian in a 
rather different sense
than that defined by the three desiderata of section~\ref{SectionBrouwerCriteria},
and one that would have been of interest to Weyl,
was proposed and enthusiastically defended
by Georges Reeb 
at the University of Strasbourg.%
\footnote{For Reeb’s 
(philosophy of) 
nonstandard analysis,
see,
in French,
\citealt{Reeb1979, 
Reeb1981, 
Barreau.Harthong1989,
Diener.Reeb1989,
Lobry1989,
Reeb.Harthong1989,
LOuvert1994},
and
\citealt{Salanskis1999}.
There is not much about Reeb’s 
(philosophy of) 
nonstandard analysis in English; 
see  
\citealt{FletcherHrbacekKanoveiKatzLobrySanders2017}
for a few recent remarks.}

The mathematical content of nonstandard analyis
as advocated by him is that of his friend Nelson’s IST,%
\footnote{Nelson has written:
\begin{quotation}
One of the most treasured experiences of my life is my friendship
with Georges Reeb. 
We had many strong discussions together, intuitionist
versus formalist. 
What he created was unique in my experience. 
His rare
spirit, 
gentle but fiercely demanding of the highest standards, 
inspired a group of younger mathematicians with an unmatched ethos of
collegiality. 
And their discoveries are extraordinary.

Reeb found, 
and led others to find, 
not only knowledge and beauty
in mathematics, 
but also virtue.
His insights into the nature of
mathematics will point the way towards the mathematics of the future.
\citep[p.8]{Nelson1996}
\end{quotation}
}
described in section~\ref{SectionClassicalConstructive}.
The originality of Reeb’s approach lies in
the fact that,
instead of construing IST as an axiomatic theory,
in which the predicate 
‘\st’
is taken to be implictly defined
by the axioms,
he proposes a specifically Brouwerian motivation for accepting 
IST as a formal theory.
It is,
he notes,
a train of thought that starts from intuitionistic observations on ZFC and
concludes to ‘the plausibility or the naturalness of IST’ 
\citep[p.151]{Reeb1989};%
\footnote{In 
\citealt[p.153]{Reeb1981},  
he had stated that
his notion of 
naïve numbers leads to ideas that
‘show some analogy with IST’,
and this is what one expects of a motivation in constructive reality
of a distinction in an idealised,
classical formal theory.
Note that Reeb in his writings does not much discuss his philosophical differences with Nelson.
On Nelson’s philosophy of mathematics,
see,
besides his own papers,
also \citealt{Buss2006}.}
and the article he dedicated to giving
his most elaborate account of this,
written with Jacques Harthong,
was,
indeed,
titled
‘Intuitionnisme 84’
\citep{Reeb.Harthong1989}.%
\footnote{1984 is the year in which a first version was written and began to circulate.}
The title was,
of course,
at the same time 
a reference to Robinson’s ‘Formalism 64’ \citep{Robinson1965}.
Just as Robinson asked what formalism could be in 1964,
Reeb and Harthong had a view on what intuitionism could be in 1984.
In addition,
formalism is  an essential component of
Reeb’s approach,
but with a specifically Brouwerian view on it.
On the other hand,
Reeb 
was not a Brouwerian intuitionist,
and did not aspire to be.
Reeb and Harthong  acknowledge the difference
when they speak of
‘the intuitionistic conception (ours just as much as that of Brouwer) …’
\citep[p.52]{Reeb.Harthong1989}.%
\footnote{‘la conception intuitionniste (aussi bien la nôtre que celle de Brouwer) …’}
Also Harthong in his afterword to the 1989 reprint of
Int84 in La mathématique nonstandard
is quite forthcoming on this point \citep[pp.265–267]{Harthong1989}.
Some differences will be touched upon below.

Reeb holds that there is a mathematical reality,
constructive,
indepent of theory,
and which intuitionists aim to describe;%
\footnote{This attitude was later described by Sundholm and myself
as 
the intuitionists’ 
‘ontological descriptivism’,
an attitude they share with Platonists,
the disagreement being over the nature of that reality
\citep[p.71]{Sundholm.Atten2008}.
If we had known the paper by Reeb and Harthong then,
we would surely have taken it into account.}
formal theories such as ZFC and IST 
as constructive objects that
codify
ideal(ised) theories of that reality,
just as we have idealised theories of physical reality
\citep[sections 2 and 5]{Reeb.Harthong1989}.
Accordingly,
the notion of motivation for an axiom takes on a richer sense
that relates the formal axiom to mathematical reality.
That relation need not be as strong as the axiom being
fully interpretable in that reality;
it may rather be construed as an idealisation.
This sounds Hilbertian,
and it is,%
\footnote{Besides the main influence Brouwer, in Reeb one finds quotations or echos from for example Hilbert,
Poincaré, Löwenheim, Skolem, and Von Neumann.}
but it must be remembered
that it was intuitionistic criticism of his earlier program
that led
Hilbert to adopt that particular view,
which itself goes back to Brouwer’s dissertation
\citeyearpar{Brouwer1907}.%
\footnote{Brouwer makes the point in \citealt{Brouwer1928A2}.}
The terms of the formal theory do, as such, not refer,
but if we construe that theory as an idealisation of reality
we must be prepared to act as if they refer to ideal(ised) objects \citep[p.30]{Salanskis1994}.
Consistency or conservativeness of the formal theory is therefore not the whole criterion:
if ZFC is extended with an independent proposition 
\(P\),
or,
alternatively,
\(\neg P\),
Reeb expresses a preference for the one that seems in a sense
closer to mathematical reality than another:
writing about Fermat’s Last Theorem in 1989,
he says that if it turns out to be an undecidable proposition,
one could of course add its negation to ZFC,
but 
‘the intuitionist will consider this choice … far removed
from concrete reality’ 
–~in which by then no counterexample had been found
\citep[p.158]{Reeb1989}.

The importance of Brouwer
in Reeb’s view
is epistemological,
and defined by the insistence
that a formal theory,
even if shown consistent,
cannot,
once constructed,
replace mathematical reality 
\citep[section 4]{Reeb.Harthong1989}, 
and that,
in particular,
accepting the Principle of the Excluded Middle
in the formal theory does not mean
that every problem in mathematical reality
can be solved.
Reeb’s main reference for this is
‘Intuitionistic reflections on formalism’
of 1928,
in which Brouwer writes:
\begin{quotation}\label{BrouwerInsights}
The disagreement over which is correct,
the formalistic way of founding mathematics anew
or the intuitionistic way of reconstructing it,
will vanish,
and the choice
between the two activities be reduced to a matter of taste,
as soon as the following
insights,
which pertain primarily to formalism but were first
formulated in the
intuitionistic literature,
are generally accepted.
The acceptance of these insights is
only a question of time,
since they are the results of pure reflection and hence
contain
no disputable element,
so that anyone who has once understood them must accept
them.
Two of the four insights have so far been understood and
accepted in the
formalistic literature.
When the same state of affairs has been reached with respect
to the other two,
it will mean the end of the controversy concerning the
foundations
of mathematics.

\emph{First insight.}
The differentiation,
among the formalistic endeavors,
between a
construction of the ‘inventory of mathematical formulas’
(formalistic view of mathematics)
and an intuitive
(contentual)%
\footnote{[The role of Brouwer’s ‘contentual’ mathematics’
corresponds to that of Reeb’s ‘mathematical reality’;
but  the former is richer than the latter.]}
theory of the laws of this construction,
as well as the
recognition of the fact that for the latter theory the
intuitionistic mathematics of the set of
natural numbers is indispensable.

\emph{Second insight.}
The rejection of the thoughtless use of the logical
principle of
excluded middle,
as well as the recognition,
first,
of the fact that the investigation of the
question why the principle mentioned is justified and to
what extent it is valid constitutes
an essential object of research in the foundations of
mathematics,
and,
second,
of the fact
that in intuitive
(contentual)
mathematics this principle is valid only for finite systems.

\emph{Third insight.}
The identification of the principle of excluded middle with
the
principle of the solvability of every mathematical problem.

\emph{Fourth insight.}
The recognition of the fact that the
(contentual)
justification of
formalistic mathematics by means of the proof of its
consistency contains a vicious circle,
since this justification rests upon the
(contentual)
correctness of the proposition that from
the consistency of a proposition the correctness of the
proposition follows,
that is,
upon the
(contentual)
correctness of the principle of excluded middle.
(\citealt[p.375]{Brouwer1928A2},
trl.~\citealt[pp.490–491]{Heijenoort1967})
\end{quotation}

The distinction in the First insight Brouwer
had made first in \citealt{Brouwer1908C}, 
in which he had shown that
\(\neg\neg(P \vee \neg P)\),
and in 1924 this led him to comment on 
Hilbert’s aim of a 
consistency proof for formalised classical mathematics
that ‘We need by no means despair of reaching this goal’
(\citealt[p.3]{Brouwer1924N}, 
trl.~\citealt[p.336]{Heijenoort1967}). 
In 
‘Intuitionistic reflections on formalism’
he added
a proof that finite conjunctions of instances of PEM are also consistent,
and in his
first Vienna lecture 
he voiced the expectation that
‘An appropriate mechanization of the language of this intuitionistically non-contradictory 
mathematics should therefore deliver precisely what the formalist school has set as its goal’ 
(\citealt[p.164, trl.~mine]{Brouwer1929A}).

Against this background,
the question that led Reeb
to embrace IST is
(as I formulate it) 
the following.
If
for example
the classical logic in the formal theory is taken to be an idealisation of the constructive logic 
of mathematical reality,
is there, 
similarly, 
an aspect of constructive reality that,
when idealised,
would lead to the notion 
of a standard number in the formal theory?
What is asked for is an intuitive motivation,
for introducing
the predicate ‘\st’ in the idealised theory IST,
not a formal proof 
(of the existence of formal nonstandard models).

Reeb answers that
the standard numbers in a formal theory may be seen as an
idealisation of what he calls the naïve whole numbers 
in reality
(‘\french{les entiers naïfs}’).
They are the numbers
‘that exist independently from the theory one uses to describe them’ 
\citep[p.63]{Reeb.Harthong1989},
and are obtained by just putting units together,
Reeb’s 
(necessarily informal)
definition is
(\citealt[p.277, p.286n3]{Reeb1979};
\citealt[p.152]{Reeb1989}):
\begin{enumerate}
\item \(0\) is naïve; 
\item if \(n\) is naïve, so is \(n+1\);
\item No other \(n\) is naïve.
\end{enumerate}
Although  the property of being naïve is not in any sense vague,
so there is no threat of Wang’s Paradox,%
\footnote{Wang’s Paradox is: 
\(0\)
is a small number;
if
\(n\)
is a small number,
so is
\(n+1\);
therefore,
all numbers are small.
This has generated quite some discussion;
the classical papers are  
\citealt{Dummett1975}
and
\citealt{Wright1975}.}
Reeb resists
the argument by induction that the naïve numbers form a set in the classical sense,
as we will see in a moment.
It should also be noted that Reeb’s
position is dissociated  from finitism:
a naïve number may be constructed in an algorithm or program
of any complexity
\citep[section 15]{Reeb.Harthong1989}. 
It is implied,
then,
that there is also a naïve notion of algorithm.
This coincides with the idea
that formal recursion theory,
if to be understood as a theory of computability,
presupposes
such a pre-theoretical notion.%
\footnote{Briefly,
the point is that a recursive function is defined by a set of equations,
and if the function is to be considered as computable,
there must be an effective method to determine that set;
but now to understand 
‘effective’ as ’ recursive’ would be circular.
A detailed presentation is given 
in
\citealt[pp.~340-342]{Heyting1958A}.
For discussion and further references, see 
\citealt{Coquand2014} 
and
\citealt{Sundholm2014}.}

At this point Reeb invokes Brouwer’s weak counterexamples to classical theorems.
Assume the consistency of the formal theory
and 
suppose that there is a predicate
\(A\)
such that
\begin{itemize}
\item the formal theory proves \(\exists x \neg A(x)\),
or it can be shown that \(\exists x \neg A(x)\) is independent and we are willing to add it as an axiom;
\item \(A\) and its evaluation at a naïve \(n\) can be understood in naïve terms,
\item but we do not have yet a naïve construction for such a counterexample.
\end{itemize}
In that case ZFC proves the formal existence of a natural number
for which we do not have a construction in reality yet.
It is a formal number in the formal set 
\(\numnat\)
to which corresponds no naïve number in reality
(in any case not yet),
and which is greater than any naïve number interpreted in the theory.
Hence Reeb’s slogan
Q: ‘The naïve whole numbers do not fill \(\numnat\)’
(‘\french{Les naifs ne remplissent pas \(\numnat\)}’).
At first,
Reeb called Q 
an observation
(‘\french{constat}’),
later a slogan
(‘\french{slogan}’).
As  Salanskis points out,
the latter is much more appropriate,
as seeing things the way Reeb does
is not a matter of direct perception
but requires accepting a certain philosophical view
\citep[p.29]{Salanskis1994}.
Furthermore,
it requires the conviction that there always will be such predicates 
\(A\)
that can be understood and evaluated in naïvely.

Reeb supposed that Fermat’s conjecture
provided such a predicate \(A\),
writing,
for example,
in 1989:
\begin{quotation}\label{Fermat}
Consider the unique object
\(a\)
in
\(\numnat\)
defined by the formula
(not well formalised,
but the reader will know how to write a perfect formula):
\begin{quote}
If the statement known as Fermat’s Great Theorem
is true,
\(a=0\),
otherwise
\(a=x+y+z+n\),
where
\(n>3, x^ + y^n = z^n\),
\(x,y,z>1\),
and
\(x,y,z\)
and
\(n\)
chosen such that
\(a\)
is as small as possible.
\end{quote}
At the moment I am writing this,
it is not possible to convince oneself
that
\(a\)
is naïve.
\citep[p.152, trl.~mine]{Reeb1989} 
\end{quotation}
It follows from Wiles’ proof,
published in  1995,
that
\(a=0\);
one may think of Goldbach’s conjecture instead.
At times,
Reeb did not invoke potential examples and limited his motivation 
to pointing out that 
the existence of such predicates
\(A\)
cannot be excluded.%
\footnote{Personal communication from Jean-Michel Salanskis,
who was a member of Reeb’s group.}
Of course,
should a naïve number 
\(n\)
be found such that
\(\neg A(n)\),
one looks for another
predicate of that type.%
\footnote{For Reeb, 
a constructive proof can exist without having been found:
‘ou bien il y a une démonstration constructive, déjà connue ou non …’
\citep[section 16]{Reeb.Harthong1989}.
For Brouwer,
on the contrary,
the only sense in which a proof can be said to exist 
is that it has been constructed.
However,
for the matter at hand
this makes no difference.}

The step to IST is made by 
idealising the distinction
between
the (constructive) naïve numbers
and the 
surplus of formal numbers in
the formal set \(\numnat\),
whose existence is expressed in Q,
to that between the classical standard ones
and the nonstandard ones.%
\footnote{%
\label{istnotunique}This idealisation need not lead to IST;
it was the theory Reeb knew and liked,
but closely related axiomatic nonstandard theories have been developed in the meantime
\citep{Kanovei.Reeken2004}.
Just as in the natural sciences,
different theories of the same phenomena in reality
may be developed,
and have different theoretical virtues.}
Given the differences between the notions of
constructive existence in reality and
formal existence in a classical theory,
it is only to be expected that the idealisation
will be one by analogy:
\begin{quotation}
It is now a matter of drawing up a suitable list of 
simple and formal
properties verified 
(or simply suggested)
by the naïve numbers,
and to consider the formal theory
consisting of the statements on this list together with the axioms of classical mathematics.

The theory known by the abbreviation IST
developed by E.~Nelson
realises this program efficiently.
But like every formal theory,
it does not escape observation Q
(i.e.,
‘The naïve numbers do not fill the standard whole numbers of IST’).
\citep[p.287, trl.~mine]{Reeb1979}
\end{quotation}
To illustrate that last remark:
the number
\(a\)
defined in terms of Fermat’s Last Theorem
exists classically
and is unique,
and hence,
as was clear also before Wiles’ demonstration, 
a standard object in IST
(Theorem~\ref{thm-uniqueness}, page \pageref{thm-uniqueness}).

In
‘\french{La mathématique non standard vieille de soixante ans?}’,
Reeb presents a beginning of such a list of properties of the naïve numbers,
which,
somewhat abbreviated,
runs as follows
\citep[pp.278–279]{Reeb1979}:%
\footnote{As Salanskis emphasises \citep[p.140]{Salanskis1999}, 
Reeb writes ‘properties’, not ‘theorems’,
so as to distinguish assertions about reality from provable formulas in a formal system.}\vspace{\baselineskip}

\noindent
Let 
\(ω\)
be a fixed,
non-naïve number.

\begin{itemize}
\item
Property 1. If
\(a\)
is naïve,
then
\(ω > a\).

\item
Property 2.
\( \dots, ω-a, \dots, ω-2, ω-1, ω, ω+1, ω+2, \dots, ω+a, \dots\)
(where
\(a\)
is naïve)
are non-naïve elements of \(\numnat\).
Likewise,
the following elements of
\(\numnat\)
are non-naïve
(where
\(a\)
is naïve):
\(ω^2,
ω^3,
\dots,
ω^a,
\dots,
\left[\tfrac{ω}{a}\right],
\sqrt[\raisebox{1pt}{\scriptsize\textit{a}}]{ω} 
\)
where
\(
[x]
\)
stands as usual for the whole part of
\(x\);
\(p_ω\),
prime number and 
\(p_ω > ω\)
(such 
\(p_ω\)
certainly exist).
The number
\( ω! \)
is not naïve and has
every naïve whole number
as divisor.

\item
Property 3.
If
\(a > 1\)
is naïve,
then
\(a^ω > ω^a\).

\item
Property 4. There exists no set
\(X\)
such that
‘\(x \in X\)’
is equivalent to
‘\(x\) is a naïve whole number’.

\item
Property 5. Let
\(X \subset \numnat\)
be a set such that 
\(n \in X\)
for every naïve 
\(n\)
[respectively,
\(ω \in X\)
for every non-naïve 
\(ω \)].
Then
there exists a non-naïve
\(α\)
such that
\(α \in X\)
[respectively,
a naïve
\(a\)
such that
\(a \in X\)].

\item
Property 6.
If
\(X\)
is a set of which every element is a naïve whole number,
then 
\(X\)
is finite.
\end{itemize}

Reeb’s argument for Property 4 is that,
if the naïve whole numbers 
formed a set,
then by induction this set would
be identical to
\(\numnat\),
and that would contradict
Slogan Q.

If in this list one replaces 
‘(non-)naïve number’
with
‘(non-)standard number’,
and construes these statements not contentually
but formally,
one gets  theorems of IST.

Reeb also notices 
(current)
limitation of this motivation
\citep[p.287]{Reeb1979}: 
\begin{enumerate}

\item He should like to have a notion of 
‘naïve object’ that extends beyond natural numbers.
In the formal counterpart,
IST,
the predicate
‘standard’
can  be meaningfully 
applied to any set,
hence to any object in its universe
(and thereby yield either a truth or a falsehood);
but Reeb does not have a correspondingly general
notion of construction.
(Note that Reeb was aware of,
but does not embrace,
Brouwer’s wider notion of constructivity.)

\item He has not been able to find a justification
for the claim that whenever all naïve whole numbers
have a certain internal property,
all numbers in
\(\numnat\)
have it.
(An analogue to Transfer.)

\item Likewise,
while for a given naïve
function such as 
\(e^x\)
it is easily shown that
an infinitesimal increase in the argument
leads to an infinitesimal increase in the value,
it remains to be shown that this is equivalent to 
(\(ε-δ\)-)continuity of the function,
which would require an analogue to Standardization.

\end{enumerate}

But Transfer and Standardization 
are,
in their full generality, 
by and large nonconstructive;
see  the fine-grained analysis by Sanders in 
\citealt{Sanders2017}.
Although Sanders’ analysis is concerned with relations of the formal standard objects
with the nonstandard ones,
and not with Reeb’s naïve objects,
his results strongly suggest that such justifications as Reeb hoped to find
will not be forthcoming.
That is far from saying that his attempt to find a natural way into
IST fails;
but it does mean that the idealisations involved in
moving from the distinction between naïve and non-naïve numbers
to IST are stronger than perhaps was expected.%
\footnote{This last consideration is one among several that leads to the question
of constructive analogues to IST
(which was not a particular concern to Reeb,
to whom,
on the contrary,
the idea of using a classical formal theory was attractive).
For this,
I refer to the papers mentioned in footnote~\ref{constructivepapers}.}

To make the transition from the naïve numbers in reality 
to the standard numbers in IST more explicit,
Reeb introduced a middle term,
‘Naïve’ 
(\citealt[pp.453–454]{Reeb1981}; \citealt[pp.153–154]{Reeb1989}). 
This term applies to all objects 
whose existence 
in ZFC is established by proving
a formula of the form
\[
\exists!x A(x)
\] 
This includes all naïve whole numbers,
but much more:
\(\numnat\), 
\(\numrat\),  
\(\numreal\), 
\(P(\numnat)\), 
\(\text{exp}\), 
\(\text{sin}\), 
\(π\), 
…
The key principle then would be: ‘There exists a finite set 
\(F\) 
that contains every
Naïve
object.
One could even decide to do nonstandard analysis
using this informal concept.
Reeb remarks that,
on the one hand,
the advantage of doing this is allow one to reconstruct
non-standard analysis in a way of which
the consistency and conservativeness are evident;
on the other hand,
as an informal concept
Naïve
may be more difficult to work with than Nelson’s formal theory
\citep[p.154]{Reeb1981}.
Moreover,
it would require quite a sophisticated argument to justify the key principle
\citep[p.154]{Reeb1989}.
His conclusion is that
‘in this sense,
the formalized theory IST
is superior to our consideration of Naïve objects,
whose interest is limited to the didactical or heuristic sphere’
\citep[p.154]{Reeb1981}.
It seems to me that,
as a motivation for the introduction of a distinction in an idealised formal theory,
a good heuristic will fit the bill.

As is clear from
the list of four
‘insights’,
the idea that one may simultaneously accept 
(not just finitary but even)
intuitionistic mathematics as
true and formalized classical mathematics as consistent was
shared by Brouwer,
who at the time was even optimistic about the prospects
of a formal consistency proof.
Yet, 
Brouwer would not have called non-standard analysis in the form of IST
an idealised formal theory of the mathematical reality 
that is the intuitive continuum.
After all,
IST
is
a syntactical enrichment of 
ZFC and in particular of the theory of the classical real numbers,
but is not an ontological enrichment of the latter.
The objects of IST are the classical real numbers.
However,
those can not be construed as idealisations of intuitionistic choice sequences.
This is clearest from the mathematical contrast
provided by
Brouwer’s strong counterxamples,
which
show that the intuitive continuum
of mathematical reality,
analysed in terms of choice sequences,
has properties that in classical analysis 
with its discrete continuum
are contradictory.
(As we saw,
IST is, 
on the contrary,
conservative over ZFC.)
Illustrative are the following:
\begin{theorem}
\(\neg \forall x\in \numreal(x \in  \numrat \vee  x \not\in  \numrat)\) \textup{\citep{Brouwer1927B}}
\end{theorem}

\begin{theorem}
\(\neg \forall x\in \numreal(\neg \neg  x > 0 \rightarrow  x > 0)\)  \textup{\citep{Brouwer1949A}}
\end{theorem}

\begin{theorem}
\(\neg \forall x\in \numreal(x \neq  0 \rightarrow  x < 0 \vee  x > 0)\) \textup{\citep{Brouwer1949B}}
\end{theorem}

\noindent
For Reeb
these theorems are not relevant,
as these they depend on intuitionistic considerations
that go beyond the finitary mathematics
he accepts as mathematical reality.

\section{Weyl and infinitesimal analysis}

Overall,
one suspects that to Brouwer,
the glass that Reeb offers would
have seemed to be half empty,
what with its essential involvement of a formalism
and its limitation of mathematical reality to the finitary.

For Weyl this would probably have been different.
As is well known,
Weyl  acknowledged the
epistemological superiority of 
intuitionism
in pure
mathematics:
\begin{quote}
With Brouwer, mathematics gains the highest intuitive clarity;
his doctrine is the culmination of idealism in mathematics.
\end{quote}
However,
Weyl continues:
\begin{quote}
But with pain the mathematician sees the greater part of his
high-rising theories dissolve into the fog.
\citep[p.24]{Weyl1925}
%
\end{quote}
Weyl had come to convince himself that it is necessary
to abandon the intuitionistic program
because he took it to be a fact that
intuitionistic mathematics is not able to found 
the mathematics required in physics,
whereas
‘mathematics should put itself to the service of the natural sciences’ 
\citep{Weyl1926}.%
\footnote{For a detailed account of that episode,
see for example 
\citealt[section 6.2.1]{Mancosu.Ryckman2002}  
and 
\citealt[section 7]{Tieszen2000b}.
A recent philosophical discussion on constructive mathematics in physics
is \citealt{Ardourel2012}.}
This pushed Weyl towards a formalist foundation
of classical mathematics.%
\footnote{But,
as we have seen
(the four ‘insights’, p.~\pageref{BrouwerInsights}),
Brouwerian intuitionism does not exclude
a formalist foundation
of classical mathematics;
it includes it as a proper part.
However,
it is not the part of intuitionistic mathematics
that is concerned with the development of
contentual mathematics;
and the contentual mathematics that Brouwer
sought to develop is far richer than the
minimum required to get the formalist
foundation going.}

In various physical
contexts it is,
conceptually, 
natural to apply nonstandard analysis,
for example when phenomena are involved at
greatly different scales, 
or
where the difference between the observable and the unobservable
plays a role.%
\footnote{Given the properties of human vision,
even at an everyday scale
infinitesimal analysis can be useful,
as shown by the analysis of the moiré effect in \citealt{Harthong1981}.
Further applications are presented in,
e.g.,
\citealt{Cutland1988},
\citealt{ArkerydCutlandHenson1997},
and
\citealt{LobrySari2008}.
A recent view from a philosopher of science is
\citealt{Wenmackers2016}.}
In a more foundational spirit,
Robert  
notes that,
although tangent vectors can be said to represent infinitesimal displacements at a point in a differentiable manifold,
it would be important to have further analyses of
differentiability
and of continuity,
neither of which can be 
based on the concept
of a differential;
a theory of infinitesimals would be one,
and moreover
supply an algebra of differentials \citep[p.xiv]{Robert1988}.

That last remark brings us to Weyl’s infinitesimal geometry in his theory of spacetime,
and also to the quotation from Weyl with which this paper begun.
For,
as Laugwitz aptly observed,
\begin{quote}
And even if Hermann Weyl declared the infinitesimals to have been eliminated,
his book
\textit{Space-Time-Matter},
widely available in several editions since 1919,
is a perfect example
of infinitesimal mathematics in action.
\citep[p.241, trl.~mine]{Laugwitz1986}
\end{quote}
Weyl’s
infinitesimal geometry
is not,
in fact,
constructive.
The fundamental notions are introduced axiomatically,
instead of being constructed out of basic intuition,
and the proof of uniqueness of Pythagorean metric of 1922 is,
in its dependence on the excluded middle,
classical.%
\footnote{On Weyl’s non-constructive mathematics in physics,
see
\citealt[p.146]{Weyl1922}; 
\citealt[p.7]{Weyl1988};
\citealt[pp.95–97]{Scholz2001}; 
and
\citealt[pp.277, 608–610, 777–778]{Eckes2011}.}

Brouwer,
no less of an idealist than Weyl,
always was a conventionalist about the structure of
physical space:
‘a question of convenience, of taste, or of custom’,
he wrote in 1909,
in a paper to which in a reprint of 1919 he added a note
stating that the general theory of relativity
‘would not affect the conclusions on the theory of knowledge’ he had reached
(\citealt[p.14]{Brouwer1909A}, trl.~\citealt[p.116]{Brouwer1975}; \citealt[p.vi]{Brouwer1919B}, trl.~\citealt[p.120]{Brouwer1975}).

Weyl was well aware of the discrepancy
between his constructivist philosophy of pure mathematics
and his classical practice in mathematical physics.
Did Weyl ever hope to give his infinitesimal
geometry a constructive foundation later
–~Brouwer had taken that attitude towards his own theorems in
classical topology~–, %
\footnote{%
In a retrospective remark of 1920
Brouwer wrote that,
when he had just begun to develop intuitionism,
‘in my contemporary 
philosophy-free mathematical
papers I have frequently also used the old [i.e., non-intuitionistic]
methods, 
trying however to derive only such results
as could be hoped to find,
after the completion
of a systematic construction of intuitionistic set theory,
a place in the new system and claim a value,
perhaps in modified form.’
\citep[p.204, trl.~mine]{Brouwer1920J}}
and can a notion of subjectivity be motivated
that supports,
as Weyl’s philosophical foundation of physics requires,
the idea of 
a subject located in a point and whose intuitive space is of
infinitesimal size%
\footnote{See in particular 
\citealt[pp.~246–248]{Bernard2013},
and Bernard’s instructive,
unpublished manuscript
\
{Bernard}.} 
without,
at the same time,
idealising beyond a notion of subject
appropriate for
constructive mathematics?
To investigate these questions would go beyond the scope of the present
paper.
But if there is no construction allowing to treat infinitesimals as individual objects,
then
to Weyl,
Reeb’s approach,
what with its
combination of a classical formalism for nonstandard analysis 
and a nevertheless intuitionistic epistemology,
might have made the glass seem at least half full.%
\footnote{Palmgren indicates that Nelson’s nonstandard analysis,
which corresponds to part but not all of Robinson’s,
may well lend itself to constructivisation
\citep[p.234]{Palmgren1998}.
Weyl,
on the other hand,
would presumably have been interested in the classical formalism.}

\paragraph{Acknowledgement.}
Earlier versions of this paper were presented at the conference
‘Weyl and the Problem of Space: From Science to Philosophy’,
University of Konstanz,
May 2015,
and at
the workshop
‘Workshop on the Continuum in the Foundations of Mathematics and Physics’,
University of Amsterdam,
April 2017.
I am grateful to the organisers for their invitations,
and to the audiences for their
questions and comments.
I have also benefited from
exchanges with
Julien Bernard
(who also shared his
instructive,
unpublished manuscript
‘New insights on Weyl’s Problem of Space,
from the correspondence with Becker’ with me), 
Dirk van Dalen,
Bruno Dinis,
Mikhail Katz,
Carlos Lobo, 
David Rabouin,
Jean-Michel Salanskis,
Sam Sanders,
Wim Veldman,
and Freek Wiedijk.
Gödel’s shorthand notes on 
the non-Archimedean number system
in Brouwer’s dissertation,
mentioned in footnote~\ref{fngoedel},
were kindly transcribed by Eva-Maria Engelen.
These notes are owned by the Institute for Advanced Study
and kept in the Department of Rare Books and Special Collections
at the Firestone Library,
Princeton University.%
\footnote{Scans are avialable online at 
\url{https://library.ias.edu/godelpapers},
and a full transcription of one notebook (vol.~X)  at
\url{https://halshs.archives-ouvertes.fr/hal-01459188}}


\begin{thebibliography}{105}
\providecommand{\natexlab}[1]{#1}
\providecommand{\url}[1]{\texttt{#1}}
\expandafter\ifx\csname urlstyle\endcsname\relax
  \providecommand{\doi}[1]{doi: #1}\else
  \providecommand{\doi}{doi: \begingroup \urlstyle{rm}\Url}\fi

\bibitem[Ardourel(2012)]{Ardourel2012}
V.~Ardourel.
\newblock \french{La physique dans la recherche en mathématiques
  constructives}.
\newblock \emph{Philosophia Scientiae}, 16\penalty0 (1):\penalty0 183--208,
  2012.

\bibitem[Arkeryd et~al.(1997)Arkeryd, Cutland, and
  Henson]{ArkerydCutlandHenson1997}
L.~O. Arkeryd, N.~J. Cutland, and C.~W. Henson, editors.
\newblock \emph{Nonstandard Analysis. Theory and Applications}.
\newblock Springer, Dordrecht, 1997.

\bibitem[{van Atten}(2004)]{Atten2004}
M.~{van Atten}.
\newblock \emph{On Brouwer}.
\newblock Wadsworth, Belmont, 2004.

\bibitem[{van Atten}(2015)]{Atten2015}
M.~{van Atten}.
\newblock \emph{Essays on Gödel’s Reception of Leibniz, Husserl, and
  Brouwer}.
\newblock Springer, 2015.

\bibitem[{van Atten}(forthcoming)]{Atten-forthcoming}
M.~{van Atten}.
\newblock The Creating Subject, the Brouwer-Kripke Schema, and infinite proofs.
\newblock Forthcoming in \emph{Indagationes Mathematicae}.


\bibitem[{van Atten} and Sundholm(2017)]{Atten.Sundholm2017}
M.~{van Atten} and G.~Sundholm.
\newblock \english{L. E. J.~Brouwer's “Unreliability of the logical
  principles”. A new translation, with an introduction}.
\newblock \emph{History and Philosophy of Logic}, 38\penalty0 (1):\penalty0
  24--47, 2017.

\bibitem[van Atten et~al.(2002)van Atten, van Dalen, and
  Tieszen]{Atten.Dalen.Tieszen2002}
M.~van Atten, D.~van Dalen, and R.~Tieszen.
\newblock \english{Brouwer and Weyl: the phenomenology and mathematics of the
  intuitive continuum}.
\newblock \emph{Philosophia Mathematica}, 10\penalty0 (3):\penalty0 203--226,
  2002.

\bibitem[Baron(1969)]{Baron1969}
M.~Baron.
\newblock \emph{The Origins of Infinitesimal Calculus}.
\newblock Pergamon Press, Oxford, 1969.

\bibitem[Barreau and Harthong(1989)]{Barreau.Harthong1989}
H.~Barreau and J.~Harthong, editors.
\newblock \emph{\french{La Mathématique non standard}}.
\newblock Éditions du CNRS, Paris, 1989.

\bibitem[van~den Berg and Sanders(2017)]{Berg.Sanders2017}
B.~van~den Berg and S.~Sanders.
\newblock Reverse mathematics and parameter-free transfer, 2017.
\newblock \url{https://arxiv.org/abs/1409.6881}.

\bibitem[van~den Berg et~al.(2012)van~den Berg, Briseid, and
  Safarik]{Berg.Briseid.Safarik2012}
B.~van~den Berg, E.~Briseid, and P.~Safarik.
\newblock A functional interpretation for nonstandard arithmetic.
\newblock \emph{Annals of Pure and Applied Logic}, 163\penalty0 (12):\penalty0
  1962--1994, 2012.

\bibitem[Bernard(2013)]{Bernard2013}
J.~Bernard.
\newblock \emph{\french{L’idéalisme dans l’infinitésimal. Weyl et
  l'espace à l'époque de la relativité}}.
\newblock Presses universitaires de Paris Nanterre, Nanterre, 2013.
\newblock Available online at \url{http://books.openedition.org/pupo/3917}.

\bibitem[Brouwer(1904--1907)]{Brouwer19041907}
L.~E.~J. Brouwer.
\newblock Notebooks, 1904--1907.
\newblock Brouwer Papers, Noord-Hollands Archief, Haarlem. Available at
  \url{http://www.cs.ru.nl/F.Wiedijk/brouwer/index.html}.

\bibitem[Brouwer(1907)]{Brouwer1907}
L.~E.~J. Brouwer.
\newblock \emph{\dutch{Over de grondslagen der wiskunde}}.
\newblock PhD thesis, \dutch{Universiteit van Amsterdam}, 1907.

\bibitem[Brouwer(1908)]{Brouwer1908C}
L.~E.~J. Brouwer.
\newblock \dutch{De onbetrouwbaarheid der logische principes}.
\newblock \emph{\dutch{Tijdschrift voor Wijsbegeerte}}, 2:\penalty0 152--158,
  1908.

\bibitem[Brouwer(1909)]{Brouwer1909A}
L.~E.~J. Brouwer.
\newblock \emph{\dutch{Het wezen der meetkunde}}.
\newblock Clausen, Amsterdam, 1909.

\bibitem[Brouwer(1917)]{Brouwer1917A2}
L.~E.~J. Brouwer.
\newblock \dutch{Addenda en corrigenda over de grondslagen der wiskunde}.
\newblock \emph{\dutch{Nieuw Archief voor Wiskunde}}, 12:\penalty0 439--445,
  1917.

\bibitem[Brouwer(1919)]{Brouwer1919B}
L.~E.~J. Brouwer.
\newblock \emph{\dutch{Wiskunde, waarheid, werkelijkheid}}.
\newblock Noordhoff, Groningen, 1919.

\bibitem[Brouwer(1920)]{Brouwer1920J}
L.~E.~J. Brouwer.
\newblock \german{Intuitionistische Mengenlehre}.
\newblock \emph{\german{Jahresbericht der deutschen Mathematiker-Vereinigung}},
  28:\penalty0 203--208, 1920.

\bibitem[Brouwer(1924)]{Brouwer1924N}
L.~E.~J. Brouwer.
\newblock {\german{{\"{U}}ber die Bedeutung des Satzes vom ausgeschlossenen
  Dritten in der Mathematik, insbesondere in der Funktionentheorie}}.
\newblock \emph{Journal f{\"{u}}r die reine und angewandte Mathematik},
  154:\penalty0 1--7, 1924.
\newblock 1923B2 in \citet{Brouwer1975}.

\bibitem[Brouwer(1927)]{Brouwer1927B}
L.~E.~J. Brouwer.
\newblock \german{{\"{U}}ber Definitionsbereiche von Funktionen}.
\newblock \emph{\german{Mathematische Annalen}}, 97:\penalty0 60--75, 1927.

\bibitem[Brouwer(1928)]{Brouwer1928A2}
L.~E.~J. Brouwer.
\newblock \german{Intuitionistische Betrachtungen \"{u}ber den Formalismus}.
\newblock \emph{KNAW Proceedings}, 31:\penalty0 374--379, 1928.

\bibitem[Brouwer(1929)]{Brouwer1929A}
L.~E.~J. Brouwer.
\newblock \german{Mathematik, Wissenschaft und Sprache}.
\newblock \emph{\german{Monatshefte f{\"{u}}r Mathematik und Physik}},
  36:\penalty0 153--164, 1929.

\bibitem[Brouwer(1930)]{Brouwer1930A}
L.~E.~J. Brouwer.
\newblock \emph{\german{Die Struktur des Kontinuums}}.
\newblock \german{{Komitee zur Veranstaltung von Gastvorträgen ausländischer
  Gelehrter der exakten Wissenschaften}}, Wien, 1930.

\bibitem[Brouwer(1948)]{Brouwer1948A}
L.~E.~J. Brouwer.
\newblock \dutch{Essentieel negatieve eigenschappen}.
\newblock \emph{Indagationes Mathematicae}, 10:\penalty0 322--323, 1948.

\bibitem[Brouwer(1949{\natexlab{a}})]{Brouwer1949A}
L.~E.~J. Brouwer.
\newblock \dutch{De non-aequivalentie van de constructieve en de negatieve
  orderelatie in het continuum}.
\newblock \emph{Indagationes Mathematicae}, 11:\penalty0 37--39,
  1949{\natexlab{a}}.

\bibitem[Brouwer(1949{\natexlab{b}})]{Brouwer1949B}
L.~E.~J. Brouwer.
\newblock \dutch{Contradictoriteit der elementaire meetkunde}.
\newblock \emph{KNAW Proceedings}, 52:\penalty0 315--316, 1949{\natexlab{b}}.

\bibitem[Brouwer(1949{\natexlab{c}})]{Brouwer1949C}
L.~E.~J. Brouwer.
\newblock Consciousness, philosophy and mathematics.
\newblock In E.~Beth, H.~Pos, and J.~Hollak, editors, \emph{{Proceedings of the
  10th International Congress of Philosophy, Amsterdam 1948}}, volume~2, pages
  1235--1249. North-Holland, Amsterdam, 1949{\natexlab{c}}.

\bibitem[Brouwer(1951)]{Brouwer1951}
L.~E.~J. Brouwer.
\newblock On order in the continuum, and the relation of truth to
  non-contradictority.
\newblock \emph{KNAW Proceedings}, 54:\penalty0 357--358, 1951.

\bibitem[Brouwer(1954)]{Brouwer1954A}
L.~E.~J. Brouwer.
\newblock Points and spaces.
\newblock \emph{Canadian Journal of Mathematics}, 6:\penalty0 1--17, 1954.

\bibitem[Brouwer(1975)]{Brouwer1975}
L.~E.~J. Brouwer.
\newblock \emph{Collected Works. Vol. 1: Philosophy and Foundations of
  Mathematics}.
\newblock North-Holland, Amsterdam, 1975.
\newblock Edited by A.~Heyting.

\bibitem[Brouwer(1981)]{Brouwer1981A}
L.~E.~J. Brouwer.
\newblock \emph{{Brouwer's Cambridge Lectures on Intuitionism}}.
\newblock Cambridge University Press, Cambridge, 1981.
\newblock Edited by D. van Dalen.

\bibitem[Buss(2006)]{Buss2006}
S.~Buss.
\newblock Nelson's work on logic and foundations and other reflections on
  foundations of mthematics.
\newblock In W.~Faris, editor, \emph{Diffusion, Quantum Theory, and Radically
  Elementary Mathematics}, pages 183--208. Princeton University Press,
  Princeton, 2006.

\bibitem[Coquand(2014)]{Coquand2014}
T.~Coquand.
\newblock Recursive functions and constructive mathematics.
\newblock In  \citet{Dubucs.Bourdeau2014}, pages 159--167.

\bibitem[Cutland(1988)]{Cutland1988}
N.~J. Cutland, editor.
\newblock \emph{Nonstandard Analysis and its Applications}.
\newblock Cambridge University Press, Cambridge, 1988.

\bibitem[{van Dalen}(1988)]{Dalen1988}
D.~{van Dalen}.
\newblock Infinitesimals and the continuity of all functions.
\newblock \emph{\dutch{Nieuw Archief voor Wiskunde}}, 6\penalty0 (3):\penalty0
  191--202, 1988.

\bibitem[{van Dalen}(1999)]{Dalen1999}
D.~{van Dalen}.
\newblock \emph{{Mystic, Geometer, and Intuitionist. The life of L. E. J.
  Brouwer. Volume 1: The Dawning Revolution}}.
\newblock Oxford University Press, Oxford, 1999.

\bibitem[{van Dalen}(2001)]{Dalen2001a}
D.~{van Dalen}.
\newblock \emph{{L.E.J.~Brouwer} en de grondslagen van de wiskunde}.
\newblock Epsilon, Utrecht, 2001.

\bibitem[{van Dalen}(2005)]{Dalen2005}
D.~{van Dalen}.
\newblock \emph{{Mystic, Geometer, and Intuitionist. The Life of L. E.
  J.~Brouwer. Volume 2: Hope and Disillusion}}.
\newblock Clarendon Press, Oxford, 2005.


\bibitem[Diener and Reeb(1989)]{Diener.Reeb1989}
F.~Diener and G.~Reeb.
\newblock \emph{\french{Analyse non standard}}.
\newblock Hermann, Paris, 1989.

\bibitem[Dinis and Gaspar(2018)]{Dinis.Gaspar2018}
B.~Dinis and J.~Gaspar.
\newblock Intuitionistic nonstandard bounded modified realisability and functional interpretation.
\newblock \emph{Annals of Pure and Applied Logic}, 169\penalty0 (5):\penalty0 392--412, 2018.

\bibitem[Dubucs and Bourdeau(2014)]{Dubucs.Bourdeau2014}
J.~Dubucs and M.~Bourdeau, editors.
\newblock \emph{Constructivity and Calculability in Historical and
  Philosophical Perspective}.
\newblock Springer, Dordrecht, 2014.

\bibitem[Dummett(1975)]{Dummett1975}
M.~Dummett.
\newblock {Wang's Paradox}.
\newblock \emph{Synthese}, 30:\penalty0 301--324, 1975.

\bibitem[Eckes(2011)]{Eckes2011}
C.~Eckes.
\newblock \emph{\french{Groupes, invariants et géométries dans l'œuvre de
  Weyl : Une étude des écrits de Hermann Weyl en mathématiques, physique
  mathématique et philosophie, 1910-1931}}.
\newblock PhD thesis, Université Jean Moulin Lyon 3, 2011.

\bibitem[Ehrlich(2006)]{Ehrlich2006}
P.~Ehrlich.
\newblock {The Rise of non-Archimedean Mathematics and the Roots of a
  Misconception I: The Emergence of non-Archimedean Systems of Magnitudes}.
\newblock \emph{Archive for the History of the Exact Sciences}, 60:\penalty0
  1--121, 2006.

\bibitem[Ferreira and Gaspar(2015)]{Ferreira.Gaspar2015}
F.~Ferreira and J.~Gaspar.
\newblock Nonstandardness and the bounded functional interpretation.
\newblock \emph{Annals of Pure and Applied Logic}, 166\penalty0 (6):\penalty0
  701--712, 2015.

\bibitem[Fletcher et~al.(2017)Fletcher, Hrba\v{c}ek, Kanovei, Katz, Lobry, and
  Sanders]{FletcherHrbacekKanoveiKatzLobrySanders2017}
P.~Fletcher, K.~Hrba\v{c}ek, V.~Kanovei, M.~Katz, C.~Lobry, and S.~Sanders.
\newblock Approaches to analysis with infinitesimals following Robinson,
  Nelson, and others.
\newblock \emph{Real Analysis Exchange}, 42\penalty0 (2):\penalty0 193--252, 2017.

\bibitem[Fraenkel(1929)]{Fraenkel1929}
A.~Fraenkel.
\newblock \emph{\german{Einleitung in die Mengenlehre. Eine elementare
  Einf\"uhrung in das Reich der unendlichen Gr\"ossen. 3te umgearbeitete und
  stark erweiterte Auflage}}.
\newblock Springer, Berlin, 1929.

\bibitem[Freudenthal(1955)]{Freudenthal1955}
H.~Freudenthal.
\newblock \german{Hermann Weyl. Der Dolmetscher zwischen Mathematikern und
  Physikern um die moderne Interpretation von Raum, Zeit und Materie}.
\newblock In \emph{\german{Forscher und Wissenschaftler im heutigen Europa.
  Weltall und Erde: Physiker, Chemiker, Erforscher des Weltalls, Erforscher der
  Erde, Mathematiker}}. Gerhard Stalling, Oldenburg, 1955.

\bibitem[Gödel(1906--1978)]{Godel.Papers}
K.~Gödel.
\newblock Papers, 1906--1978.
\newblock Department of Rare Books and Special Collections, Firestone Library,
  Princeton.

\bibitem[Hahn(1907)]{Hahn1907}
H.~Hahn.
\newblock \german{Über die nichtarchimedischen Größensysteme}.
\newblock \emph{Sitzungsberichte der Kaiserlichen Akademie der Wissenschaften,
  Wien}, 116:\penalty0 601--655, 1907.

\bibitem[Harthong(1981)]{Harthong1981}
J.~Harthong.
\newblock \french{Le moiré}.
\newblock \emph{Advances in Applied Mathematics}, 2:\penalty0 24--75, 1981.

\bibitem[Harthong(1989)]{Harthong1989}
J.~Harthong.
\newblock \french{Commentaires sur Intuitionnisme 84}.
\newblock In  \citet{Barreau.Harthong1989}, pages 253--273.

\bibitem[van Heijenoort(1967)]{Heijenoort1967}
J.~van Heijenoort, editor.
\newblock \emph{{From Frege to G\"{o}del: A Sourcebook in Mathematical Logic,
  1879--1931}}.
\newblock {Harvard University Press}, Cambridge MA, 1967.

\bibitem[Heyting(1934)]{Heyting1934}
A.~Heyting.
\newblock \emph{\german{Mathematische Grundlagenforschung, Intuitionismus,
  Beweistheorie}}.
\newblock Springer, Berlin, 1934.

\bibitem[Heyting(1958)]{Heyting1958A}
A.~Heyting.
\newblock \german{Blick von der intuitionistischen Warte}.
\newblock \emph{Dialectica}, 12:\penalty0 332--345, 1958.

\bibitem[Hilbert(1899)]{Hilbert1899}
D.~Hilbert.
\newblock \emph{\german{Grundlagen der Geometrie. Festschrift zur Feier der
  Enth\"ullung des Gauss-Weber Denkmals in G\"ottingen}}.
\newblock Teubner, Leipzig, 1899.

\bibitem[Kanovei and Reeken(2004)]{Kanovei.Reeken2004}
V.~Kanovei and M.~Reeken.
\newblock \emph{Nonstandard Analysis, Axiomatically}.
\newblock Springer, Berlin, 2004.

\bibitem[Kleene and Vesley(1965)]{Kleene.Vesley1965}
S.~Kleene and R.~Vesley.
\newblock \emph{The Foundations of Intuitionistic Mathematics, Especially in
  Relation to Recursive Functions}.
\newblock North-Holland, Amsterdam, 1965.

\bibitem[Kock(2006)]{Kock2006}
A.~Kock.
\newblock \emph{Synthetic Differential Geometry \textup{(2nd ed.)}}.
\newblock Cambridge University Press, Cambridge, 2006.

\bibitem[Kreisel(1967)]{Kreisel1967b}
G.~Kreisel.
\newblock Informal rigour and completeness proofs.
\newblock In I.~Lakatos, editor, \emph{Problems in the Philosophy of
  Mathematics}, pages 138--186. North-Holland, Amsterdam, 1967.

\bibitem[Laugwitz(1986)]{Laugwitz1986}
D.~Laugwitz.
\newblock \emph{\german{Zahlen und Kontinuum. Eine Einführung in die
  Infinitesimalmathematik}}.
\newblock BI Wissenschaftsverlag, Mannheim, 1986.

\bibitem[Leibniz(1859)]{Leibniz1859}
G.~W. Leibniz.
\newblock \emph{\german{Leibnizens mathematische Schriften}}, volume~4.
\newblock Schmidt, Halle), 1859.
\newblock Edited by C.~Gerhardt.

\bibitem[Lobry(1989)]{Lobry1989}
C.~Lobry.
\newblock \emph{\french{Et pourtant\dots\ ils ne remplissent pas N}}.
\newblock Aléas, Lyon, 1989.

\bibitem[Lobry and Sari(2008)]{LobrySari2008}
C.~Lobry and T.~Sari.
\newblock Non-standard analysis and representation of reality.
\newblock \emph{International Journal of Control}, 81\penalty0 (3):\penalty0
  517--534, 2008.

\bibitem[L'Ouvert(1994)]{LOuvert1994}
L'Ouvert.
\newblock \emph{{Numéro spécial Georges Reeb}}.
\newblock Institut de recherche sur l'enseignement des mathématiques (IREM) de
  Strasbourg, Strasbourg, septembre 1994.

\bibitem[Mancosu(1999)]{Mancosu1999}
P.~Mancosu.
\newblock \emph{Philosophy of Mathematics and Mathematical Practice in the
  Seventeenth Century}.
\newblock Oxford University Press, Oxford, 1999.

\bibitem[Mancosu and Ryckman(2002)]{Mancosu.Ryckman2002}
P.~Mancosu and T.~Ryckman.
\newblock {Mathematics and phenomenology: the correspondence between O.~Becker
  and H.~Weyl}.
\newblock \emph{Philosophia Mathematica}, Nouvelle Série, 10:\penalty0 130--202, 2002.

\bibitem[Myhill(1966)]{Myhill1966}
J.~Myhill.
\newblock {Notes towards an axiomatization of intuitionistic analysis}.
\newblock \emph{Logique et Analyse}, 9:\penalty0 280--297, 2002.

\bibitem[Nelson(1977)]{Nelson1977}
E.~Nelson.
\newblock Internal set theory: A new approach to nonstandard analysis.
\newblock \emph{Bulletin American Mathematical Society}, 83:\penalty0
  1165--1198, 1977.

\bibitem[Nelson(1986)]{Nelson1986}
E.~Nelson.
\newblock Predicative Arithmetic.
\newblock Princeton University Press, Princeton, 1986.
\newblock Available at \url{https://web.math.princeton.edu/~nelson/books/pa.pdf}.


\bibitem[Nelson(1988)]{Nelson1988}
E.~Nelson.
\newblock The syntax of nonstandard analysis.
\newblock \emph{Annals of Pure and Applied Logic}, 38\penalty0 (2):\penalty0
  123--134, 1988.

\bibitem[Nelson(1996)]{Nelson1996}
E.~Nelson.
\newblock Ramified recursion and intuitionism, 1996.
\newblock Available at
  \url{https://web.math.princeton.edu/~nelson/papers/ramrec.pdf}. The year 1996
  is that in the date of the TeX file on the same server; the original talk was
  presented to the Colloque Trajectorien, Strasbourg/Obernai, June 12-16, 1995.

\bibitem[Nelson(2002)]{Nelson2002}
E.~Nelson.
\newblock Internal set theory, 2002.
\newblock First chapter of an unfinished book on nonstandard analysis,
  available at \url{https://web.math.princeton.edu/~nelson/books/1.pdf}. The
  year 2002 is that of the pdf file on the server.

\bibitem[Palmgren(1993)]{Palmgren1993}
E.~Palmgren.
\newblock A note on mathematics of infinity.
\newblock \emph{The Journal of Symbolic Logic}, 58\penalty0 (4):\penalty0
  1195--1200, 1993.

\bibitem[Palmgren(1995)]{Palmgren1995}
E.~Palmgren.
\newblock A constructive approach to nonstandard analysis.
\newblock \emph{Annals of Pure and Applied Logic}, 73\penalty0 (3):\penalty0
  297--325, 1995.

\bibitem[Palmgren(1998)]{Palmgren1998}
E.~Palmgren.
\newblock Developments in constructive nonstandard analysis.
\newblock \emph{Bulletin of Symbolic Logic}, 4\penalty0 (3):\penalty0 233--272,
  1998.

\bibitem[Reeb(1979)]{Reeb1979}
G.~Reeb.
\newblock \french{La mathématique non standard vieille de soixante ans?},
  1979.
\newblock References are to the reprint in Appendix A to \citet{Salanskis1999}.

\bibitem[Reeb(1981)]{Reeb1981}
G.~Reeb.
\newblock \french{La mathématique non standard vieille de soixante ans?}
\newblock \emph{Cahiers de topologie et géométrie différentielle
  catégoriques}, 22\penalty0 (2):\penalty0 149--154, 1981.

\bibitem[Reeb(1989)]{Reeb1989}
G.~Reeb.
\newblock \french{0, 1, 2, etc \dots\ ne remplissent pas (du tout) N}, 1989.
\newblock Included as chapter 9 in \citet{Diener.Reeb1989}.

\bibitem[Reeb and Harthong(1989)]{Reeb.Harthong1989}
G.~Reeb and J.~Harthong.
\newblock \french{Intuitionnisme 84}.
\newblock In  \citet{Barreau.Harthong1989}, pages 213--252.
\newblock Reprinted in \citet[pp.42--77]{LOuvert1994}.

\bibitem[Robert(1988)]{Robert1988}
A.~Robert.
\newblock \emph{Nonstandard Analysis}.
\newblock John Wiley \& Sons, New York, 1988.

\bibitem[Robinson(1965)]{Robinson1965}
A.~Robinson.
\newblock Formalism 64.
\newblock In Y.~Bar-Hillel, editor, \emph{Logic, Methodology, and Philosophy of
  Science}, pages 228--246. North Holland, Amsterdam, 1965.

\bibitem[Robinson(1966)]{Robinson1966}
A.~Robinson.
\newblock \emph{Non-standard Analysis}.
\newblock North-Holland, Amsterdam, 1966.

\bibitem[Salanskis(1994)]{Salanskis1994}
J.-M. Salanskis.
\newblock \french{Un Ma{\^{i}}tre}.
\newblock In \emph{{Numéro spécial Georges Reeb}} \citet{LOuvert1994}, pages
  25--32.

\bibitem[Salanskis(1999)]{Salanskis1999}
J.-M. Salanskis.
\newblock \emph{\french{Le Constructivisme non standard}}.
\newblock Presses Universitaires du Septentrion, Villeneuve d'Ascq, 1999.

\bibitem[Sanders(2017)]{Sanders2017}
S.~Sanders.
\newblock Nonstandard analysis and constructivism!, 2017.
\newblock \url{https://arxiv.org/abs/1704.00281}.

\bibitem[Sanders(2018)]{Sanders2018} 
S.~Sanders.
\newblock To be or not to be constructive, that is not the question.
\newblock \emph{Indagationes Mathematicae}, 29\penalty0 (1):\penalty0 313--381, 2018.

\bibitem[Schmieden and Laugwitz(1958)]{Schmieden.Laugwitz1958}
C.~Schmieden and D.~Laugwitz.
\newblock \german{Eine Erweiterung der Infinitesimalrechnung}.
\newblock \emph{Mathematische Zeitschrift}, 69:\penalty0 1--39, 1958.

\bibitem[Scholz(2001)]{Scholz2001}
E.~Scholz.
\newblock {Weyls Infinitesimalgeometrie (1917--1925)}.
\newblock In E.~Scholz, editor, \emph{Hermann Weyl's Raum-Zeit-Materie and a
  General Introduction to His Scientific Work}, pages 48--104. Birkhäuser,
  Basel, 2001.

\bibitem[Schubring(2005)]{Schubring2005}
G.~Schubring.
\newblock \emph{Conflicts between Generalization, Rigor, and Intuition. Number
  Concepts Underlying the Development of Analysis in 17--19th Century France
  and Germany}.
\newblock Springer, New York, 2005.

\bibitem[Skolem(1929)]{Skolem1929}
T.~Skolem.
\newblock \german{{\"{U}}ber die Grundlagendiskussionen in der Mathematik}.
\newblock In \emph{{Den syvende skandinaviske matematikerkongress i Oslo 19--22
  August 1929}}, Oslo, 1929. Broegger.

\bibitem[Skolem(1934)]{Skolem1934}
T.~Skolem.
\newblock \german{Über die Nicht-charakterisierbarkeit der Zahlenreihe mittels
  endlich oder abzählbar unendlich vieler Aussagen mit ausschliesslich
  Zahlenvariablen}.
\newblock \emph{Fundamenta Mathematicae}, 23:\penalty0 150--161, 1934.

\bibitem[Sundholm(2014)]{Sundholm2014}
G.~Sundholm.
\newblock {Constructive recursive functions, Church's thesis, and Brouwer's
  theory of the creating subject: afterthoughts on a Parisian Joint Session}.
\newblock In  \citet{Dubucs.Bourdeau2014}, pages 1--35.

\bibitem[Sundholm and {van Atten}(2008)]{Sundholm.Atten2008}
G.~Sundholm and M.~{van Atten}.
\newblock {The proper interpretation of intuitionistic logic. On Brouwer's
  demonstration of the Bar Theorem}.
\newblock In M.~{van Atten}, P.~Boldini, M.~Bourdeau, and G.~Heinzmann,
  editors, \emph{One Hundred Years of Intuitionism (1907--2007). The Cerisy
  Conference}, pages 60--77, Basel, 2008. Birkhäuser.

\bibitem[Tieszen(2000)]{Tieszen2000b}
R.~Tieszen.
\newblock {The philosophical background of Weyl's mathematical constructivism}.
\newblock \emph{Philosophia Mathematica}, 3:\penalty0 274--301, 2000.

\bibitem[Troelstra(1982)]{Troelstra1982}
A.~Troelstra.
\newblock On the origin and development of {Brouwer's} concept of choice
  sequence.
\newblock In A.~Troelstra and D.~van Dalen, editors, \emph{The {L. E.
  J.~Brouwer} Centenary Symposium}, pages 465--486. North-Holland, Amsterdam,
  1982.

\bibitem[Troelstra and van Dalen(1988)]{Troelstra.Dalen1988}
A.~Troelstra and D.~van Dalen.
\newblock \emph{Constructivism in Mathematics}.
\newblock North-Holland, Amsterdam, 1988.

\bibitem[Veronese(1891)]{Veronese1891}
G.~Veronese.
\newblock \emph{\italian{Fondamenti di Geometria a più dimensioni e a più
  specie di unità rettilinee, esposti in forma elementare}}.
\newblock Tipografia del Seminario, Padova, 1891.

\bibitem[Vesley(1981)]{Vesley1981}
R.~Vesley.
\newblock An intuitionistic infinitesimal calculus.
\newblock In F.~Richman, editor, \emph{Constructive Mathematics}, pages
  208--212. Springer, Berlin, 1981.
\newblock Lecture Notes in Mathematics 873.

\bibitem[Wavre(1926)]{Wavre1926}
R.~Wavre.
\newblock \french{Logique formelle et logique empirique}.
\newblock \emph{\french{Revue de M\'etaphysique et de Morale}}, 33:\penalty0
  65--75, 1926.

\bibitem[Wenmackers(2016)]{Wenmackers2016}
S.~Wenmackers.
\newblock {Children of the Cosmos. Presenting a toy model of science with a
  supporting cast of infinitesimals}.
\newblock In A.~Aguirre, B.~Foster, and Z.~Merali, editors, \emph{Trick or
  Truth?}, pages 5--20. Springer, Dordrecht, 2016.

\bibitem[Weyl(1922)]{Weyl1922}
H.~Weyl.
\newblock \german{Die Einzigartigkeit der Pythagoreischen Maßbestimmung}.
\newblock \emph{Mathematische Zeitschrift}, 12:\penalty0 114--146, 1922.

\bibitem[Weyl(1925)]{Weyl1925}
H.~Weyl.
\newblock \german{Die heutige Erkenntnislage in der Mathematik}.
\newblock \emph{Symposion}, 1:\penalty0 1--32, 1925.

\bibitem[Weyl(1926)]{Weyl1926}
H.~Weyl.
\newblock \emph{\german{Philosophie der Mathematik und Naturwissenschaft}}.
\newblock Leibniz Verlag, M{\"u}nchen, 1926.
\newblock \citet{Weyl1949} is an expanded English version.

\bibitem[Weyl(1949)]{Weyl1949}
H.~Weyl.
\newblock \emph{Philosophy of Mathematics and Natural Science}.
\newblock Princeton University Press, Princeton, 1949.

\bibitem[Weyl(1988)]{Weyl1988}
H.~Weyl.
\newblock \emph{\german{Riemanns geometrische Ideen, ihre Auswirkung und ihre
  Verknüpfung mit der Gruppentheorie}}.
\newblock Springer, Berlin, 1988.
\newblock Edited by K. Chandrasekharan.

\bibitem[Wright(1975)]{Wright1975}
C.~Wright.
\newblock On the coherence of vague predicates.
\newblock \emph{Synthese}, 30:\penalty0 325--365, 1975.

\end{thebibliography}

\end{document}